\documentclass{amsart}
\usepackage[utf8x]{inputenc}
\usepackage{color}
\usepackage{float}
\usepackage{amsthm}
\usepackage{amssymb}
\usepackage{amsmath}
\usepackage{verbatim}
\usepackage{url}
\usepackage{tikz}
\usepackage{stmaryrd}

\usetikzlibrary{chains,arrows,positioning}
\usetikzlibrary{matrix,arrows}
\usetikzlibrary{shapes.geometric}

\theoremstyle{plain}
\newtheorem{theorem}{Theorem}[section]
\newtheorem{prop}[theorem]{Proposition}
\newtheorem{cor}[theorem]{Corollary}
\newtheorem{lemma}[theorem]{Lemma}

\theoremstyle{definition}
\newtheorem{defn}[theorem]{Definition}

\theoremstyle{remark}
\newtheorem{remark}{Remark}

\newcommand{\p}{{p}}

\newcommand{\PSp}{\mathrm{PSp}}

\newcommand{\Spin}{\mathrm{Spin}}

\newcommand{\PSO}{\mathrm{PSO}}
\newcommand{\SO}{\mathrm{SO}}

\renewcommand{\P}{\mathbb{P}}

\newcommand{\PD}{\mathrm{PD}}

\newcommand{\Hom}{\mathrm{Hom}}

\newcommand{\C}{\mathbb{C}}
\newcommand{\Z}{\mathbb{Z}}
\newcommand{\Q}{Q}
\newcommand{\R}{\mathbb{R}}
\newcommand{\rmd}{\mathrm{d}}

\newcommand{\F}{\mathcal W_{q,\mathrm{Lie}}}
\newcommand{\G}{\mathcal W_{q,\mathrm{Giv}}}
\newcommand{\LL}{\mathcal W_{q,\mathrm{Prz}}}
\newcommand{\W}{\mathcal W_q}
\newcommand{\HH}{\mathcal W_{q,\mathrm{Lus}}}

\newcommand{\omegaGiv}{\omega_{q,Giv}}
\newcommand{\omegaCan}{\omega_{can}}

\newcommand{\inv}{^{-1}}

\newcommand{\RR}{\check{X}_\mathrm{Lie}}
\newcommand{\X}{\check{\mathbb{X}}}
\newcommand{\TT}{\check{X}_\mathrm{Prz}}
\newcommand{\XXcan}{\check{X}_\mathrm{can}}
\newcommand{\XX}{\check{X}}
\newcommand{\XXlus}{\check{X}_\mathrm{Lus}}
\newcommand{\VV}{\check{X}_{q,\mathrm{Giv}}}

\makeatletter
\def\Ddots{\mathinner{\mkern1mu\raise\p@
\vbox{\kern7\p@\hbox{.}}\mkern2mu
\raise4\p@\hbox{.}\mkern2mu\raise7\p@\hbox{.}\mkern1mu}}
\makeatother

\title[On Landau-Ginzburg models for quadrics]{On Landau-Ginzburg models for quadrics and flat sections
of Dubrovin connections}
\author{C. Pech, K. Rietsch, and L. Williams}
\begin{document}

\begin{abstract}
This paper proves a version of mirror symmetry expressing the (small) Dubrovin connection for even-dimensional quadrics in terms of a mirror-dual Landau-Ginzburg model $(\XXcan,\W)$. Here $\XXcan$ is the complement of an anticanonical divisor in a Langlands dual quadric. The superpotential $\W$ is a regular function on $\XXcan$ and is written in terms of coordinates which are naturally identified with a cohomology basis of the original quadric.
This superpotential is shown to extend the earlier Landau-Ginzburg model of Givental, and  to be isomorphic to the Lie-theoretic mirror introduced in \cite{rietsch}. We also introduce a Laurent polynomial superpotential which is the restriction of $\W$ to a particular torus in $\XXcan$.  
Together with results from~\cite{PR2} for odd quadrics, we obtain a combinatorial model for the Laurent polynomial superpotential in terms of a quiver, in the vein of those introduced in the 1990's by Givental for type $A$ full flag varieties. These Laurent polynomial superpotentials form a single series, despite the fact that our mirrors of even quadrics are defined on dual quadrics, while the mirror to an odd quadric is naturally defined on a projective space. Finally, we express flat sections of the (dual) Dubrovin connection in a natural way in terms of oscillating integrals associated to $(\XXcan,\W)$ and compute explicitly a particular flat section.
\end{abstract}

\maketitle 
\setcounter{tocdepth}{1}
\tableofcontents

\section{Introduction}

Suppose $X$ is a smooth projective complex Fano variety of dimension $N$. Starting from $X$ as the `$A$-model', Dubrovin constructed a flat connection on a trivial bundle with fiber $H^*(X,\C)$, using  Gromov-Witten invariants of $X$, see Section~\ref{sec:connections}. The `$B$-models' of Fano varieties were first introduced in \cite{WittenPhases} and \cite{GiventalICM}. In our setting  $X$ will always have Picard rank $1$. In this case the base of the trivial bundle on the $A$-side can be taken to be the two-dimensional complex torus $\C^*_q\times\C^*_{\hbar}$ with coordinates $q$ and $\hbar$. The Dubrovin connection is flat and therefore defines a $D$-module $M_A$, where $D=\C[\hbar^{\pm 1},q^{\pm 1}]\langle\partial_\hbar,\partial_q\rangle $.

In \cite{Givental:EqGW}, Givental computed the `small $J$-function' and the `quantum differential equation' of projective hypersurfaces, such as quadrics (see Section~\ref{s:aseries}). He also proved the first mirror theorem in this setting, which states that the coefficients of the $J$-function (and hence the solutions to the quantum differential equation) can be expressed as oscillating integrals.
When the cohomology of the hypersurface is generated in degree $2$, e.g. for odd-dimensional quadrics, then the coefficients of the $J$-function generate the $A$-model $D$-module $M_A$. For even-dimensional quadrics this is no longer the case.

In this paper, we exploit the fact that quadrics are homogeneous spaces for the special orthogonal group and thus also have mirror LG models defined using Lie theory \cite{rietsch}. We express these Lie theoretic mirrors in certain canonical coordinates and show how to reconstruct in a natural way the entire $D$-module $M_A$ on the mirror side from a Gauss-Manin system $M_B$. In particular, we obtain formulas for  flat sections of the Dubrovin connection where the coefficients are oscillating integrals. We also investigate the comparison between various choices of mirrors for quadrics including particularly  Givental's mirror and our canonical LG model. 

We begin describing our results by giving an overview of various LG models for quadrics, including the new ones introduced in this paper. We are then able to state our comparison results followed by our versions of the mirror theorem and some applications.

\subsection*{Acknowledgements}
The authors thank Sasha Givental for useful comments and suggestions, 
leading to major improvements, particularly in the exposition. We thank Bernard Leclerc for pointing us to the references \cite{GLS-Partial}  and \cite{GLS-Survey}. The first two authors also thank Yank\i \ Lekili for helpful conversations. The middle author thanks Dale Peterson.  

\subsection{Overview of LG models for quadrics}

\subsubsection*{Givental's mirror.}

Givental's mirror to the quadric $X=Q_{N}$ is defined by a smooth affine variety (the Givental mirror manifold)
\begin{equation}\label{e:giventalMirror}
	\VV=\left\{(\nu_1,\dotsc, \nu_{N+2})\in (\C^*)^{N+2}\ \mid\  \prod_{i=1}^{N+2}\nu_i=q,\ \ \nu_{N+1}+\nu_{N+2}=1\right\}
\end{equation}
with superpotential  
\begin{equation}\label{e:giventalLG}
	\G(\nu_1,\dotsc, \nu_{N+2})=\nu_1+\dotsc +\nu_{N},
\end{equation}
and volume form 
\begin{equation}\label{e:giventalVF}
	\omegaGiv=\frac{\bigwedge_{i=1}^{N+2}{d\log \nu_i}}{d(\nu_{N+1}+\nu_{N+2})\wedge d\log (\prod_{i=1}^{N+2}\nu_i)}.
\end{equation}
Note that $\VV$ is a hypersurface in an $(N+1)$-dimensional torus and $\omegaGiv$ is the residue of the standard holomorphic volume form on the torus (compare e.g. \cite{Pham:book}). Givental's mirror theorem expresses the coefficients of the $J$-function of $Q_N$ as oscillating integrals involving $\G$ and $\omegaGiv$ over some middle-dimensional cycles in $\VV$.

\subsubsection*{A Laurent polynomial mirror.}

A Laurent polynomial LG model $(\TT, \LL)$ for the $N$-dimensional quadric $X=\Q_N$, 
\begin{equation}\label{e:toricLG}
	\TT=(\C^*)^N, \quad\LL = z_1+z_2+\dotsc + z_{N-1}+ \frac{(z_N+q)^2}{z_1 z_2 \cdots z_N},
\end{equation}
can be obtained from Givental's mirror by a change of variables which is essentially the one found in  \cite[Remark~19]{przyjalkowski}, see also \cite{GS}.  
We recall the change of variables in Sections \ref{comparison:odd} and \ref{comparison:even}. This LG model  is a partial compactification of Givental's mirror. The torus-invariant volume form on $\TT$ restricts to Givental's volume form~\eqref{e:giventalVF}. 

\subsubsection*{A Lie-theoretic mirror.}

The smooth quadric $Q_N$ inside $\P^{N+1}$ is naturally a homogeneous space for the group $\Spin_{N+2}(\C)$ associated to the defining quadratic form. The mirror construction from \cite{rietsch} applies in this setting. It gives a regular function $\F$ on an $N$-dimensional affine subvariety $\RR$  inside the full flag variety for the Langlands dual group, namely the full flag variety for $\PSp_{N+1}(\C)$ if $N$ is odd, and for $\PSO_{N+2}(\C)$ otherwise. The precise definition of $(\RR,\F)$ is recalled in Section~\ref{s:richardson}.

The affine variety $\RR$ also has  a holomorphic volume form $\omegaCan$, which is explicitly described in \cite{rietsch}. Indeed $\RR$ is an affine Richardson variety and it is also log Calabi-Yau as seen by combining \cite[Appendix A]{KLS} and \cite[Section~4.2]{KumarSchwede}. 

By the main result of \cite{rietsch} there is an isomorphism between the Jacobi ring of $\F$ and the quantum cohomology ring of $\Q_N$ (with the quantum parameter inverted). This is not true for the mirrors $(\VV,\G)$ and $(\TT,\LL)$.

\subsubsection*{A canonical mirror.}

The canonical mirror of an odd-dimensional quadric $\Q_{2m-1}$ was introduced in \cite{PR2}, and is defined on the complement $\XXcan$ of an anticanonical divisor in the projective space $\X=\mathbb P(H^*(\Q_{2m-1},\C)^*)$. Suppose $\p_0,\dotsc,\p_{2m-1}$ are the homogeneous coordinates on $\X$ corresponding to the Schubert basis of $H^*(\Q_{2m-1},\C)$. Then $\W:\XXcan\to\C$ is given by
	\begin{equation*}
		\W = \frac{\p_{1}}{\p_0} + \sum_{\ell=1}^{m-1} \frac{\p_{\ell+1} \p_{2m-1-\ell}}{\delta_\ell} + q\frac{\p_{1}}{\p_{2m-1}},
	\end{equation*} 
where	
	\begin{equation}\label{eq:delta}
	\delta_{\ell}  =  \sum_{k=0}^{\ell} (-1)^k \p_{\ell-k} \p_{N-\ell+k} \text{ for }1 \leq \ell \leq m-1
\end{equation}
with $N=2m-1$.
	
The canonical mirror of an even-dimensional quadric $\Q_{2m-2}$ introduced here is similar in appearance, however the mirror projective space is replaced by a `mirror quadric' $\X=\check{Q}_{2m-2}$. Note first that $\P(H^*(\Q_{2m-2},\C)^*)$ has dimension $2m-1$ and homogeneous coordinates $\p_0,\dotsc, \p_{m-1},\p_{m-1}',\dotsc \p_{2m-2}$ corresponding to the Schubert basis of $H^*(\Q_{2m-2},\C)$. The mirror quadric $\check \Q_{2m-2}$ is the quadratic hypersurface inside $\P(H^*(\Q_{2m-2},\C)^*)$ defined by 
\begin{equation*}
	\p_{m-1}\p_{m-1}'-\p_m \p_{m-2} + \dots + (-1)^{m-1} \p_{2m-2} \p_{0}=0.
\end{equation*}
The superpotential $\W $ is defined by the formula
	\begin{equation*}
		\W = \frac{\p_{1}}{\p_0} + \sum_{\ell=1}^{m-3} \frac{\p_{\ell+1} \p_{2m-2-\ell}}{\delta_\ell} + \frac{\p_m}{\p_{m-1}} + \frac{\p_m}{\p_{m-1}'} + q\frac{\p_{1}}{\p_{2m-2}},
	\end{equation*}
which is regular on the the complement $\XXcan$ of an anticanonical divisor in $\check \Q_{2m-2}$.
Here $\delta_{\ell}$ is defined by the formula in equation~\eqref{eq:delta}, with $N=2m-2$.

\subsubsection*{Laurent polynomial mirrors with a quiver description}\label{s:LaurentLikePn}

For $X=Q_{2m-1}$  the Laurent polynomial mirror
\begin{equation}\label{e:quiverMirrorOdd}
   		\HH = a_1 + \dots + a_{m-1} + c +  b_{m-1} + \dots + b_1 + q \frac{a_1+b_1}{a_1 \dots a_{m-1} c  b_{m-1} \dots b_1}. 
\end{equation} 
was introduced in \cite[Proposition 8]{PR2}. It was obtained by restricting $\F$ to a natural choice of torus $\XXlus$ in $\RR$, on which we consider coordinates like the ones used by Lusztig in \cite{Lusztig94}.

For the even quadric $X=\Q_{2m-2}$ we define here an analogous Laurent polynomial mirror
\begin{equation}\label{e:quiverMirrorEven}
     \HH=a_1 + \dots + a_{m-2} + c + d + b_{m-2} + \dots + b_1 + q \frac{a_1+b_1}{a_1 \dots a_{m-2} c d b_{m-2} \dots b_1},
\end{equation}
also obtained from a torus $\XXlus$ in $\RR$. Note that $(\XXlus,\HH)$ is not isomorphic to the other Laurent polynomial mirror $(\TT,\LL)$.

In Section \ref{sec:quiver}, we interpret $(\XXlus,\HH)$ in terms of a quiver, in the spirit of \cite{Givental:QToda,BCFKvSGr,BCFKvSPF}. The quiver we associate to $\Q_N$ looks like an augmentation of a type $D_N$ quiver (see Figure \ref{fig:quiverQN}). Note that the mirrors for type~$A$ homogeneous spaces from \cite{Givental:QToda,EHX,BCFKvSGr,BCFKvSPF} also relate to Lusztig coordinates, see \cite{rietschNagoya, rietsch}.

\subsection{Comparison of the canonical LG model with the other mirrors}

\subsubsection*{Isomorphism with the Lie-theoretic mirror}

It was proved in \cite{PR2} that for the odd-dimensional quadric $\Q_{2m-1}$ viewed as a homogeneous space for $\Spin_{2m+1}$, there is an isomorphism between the domain $\RR$ of $\F$ and the domain $\XXcan$ of the canonical mirror. This isomorphism identifies the superpotentials $\F$ and $\W$. 

\begin{theorem}[{\cite[Theorem 1]{PR2}}]
	If $X=\Q_{2m-1}$ is an odd-dimensional quadric, there is an isomorphism of affine varieties $\RR \to \XXcan$ such that the following diagram commutes
	\[
		\begin{tikzpicture}
			\matrix [matrix of math nodes,row sep=1.5cm, column sep=0.6cm,text height=1.5ex,text depth=0.25ex]
  			{
            	|(R)| \RR & & |(X)| \XXcan \\
            	& |(C)| \C &  \\
			};

			\path[->]
			(R) edge node[above,sloped,inner sep=3pt] {$\sim$} (X)
			(R) edge node[below,sloped,inner sep=3pt] {$\F$} (C)
			(X) edge node[below,sloped,inner sep=3pt] {$\W$} (C)
			;
		\end{tikzpicture}
	\]
\end{theorem}

A key ingredient in the construction of the isomorphism is the geometric Satake correspondence of \cite{lusztig,Gin:GS,MV}, which identifies the projective space $\X=\P(H^*(\Q_{2m-1},\C)^*)$ containing $\XXcan$ as the projectivisation of a representation of $\PSp_{2m}(\C)$.

In this paper, we prove the same result in the case of even-dimensional quadrics $\Q_{2m-2}$ (see Theorem \ref{t:even}).

\subsubsection*{Comparison with the Givental mirror} 

In Sections \ref{comparison:odd} and \ref{comparison:even}, we relate $(\XXcan,\W)$ to the Givental mirror $(\VV,\G)$. In particular, we prove the following proposition.

\begin{prop}\label{prop:Giv2W}
	There is an embedding, $\VV\hookrightarrow \XXcan$, of the Givental mirror manifold into the canonical mirror such that the volume form $\omegaCan$ on $\XXcan$ (suitably normalized) pulls back to $\omegaGiv$, and the superpotential $\W$ pulls back to $\G$. 
\end{prop}

An advantage of the mirror $\W$ over its predecessor $\G$ is that the former has the expected number of critical points (at fixed generic value of $q$), namely $\dim(H^*(\Q_N,\C))$. 

\begin{prop}\label{p:critical}
	The superpotential $\W:\XXcan\to \C$ for the mirror of $Q_N$ has $\dim{H^*(Q_N,\C)}$ many non-degenerate critical points. Precisely two of these in the even $N$ case, and one of these in the odd $N$ case, are not contained in the image of the embedding, $\VV\hookrightarrow \XXcan$, of the Givental mirror manifold.
\end{prop} 

In the special case of $\Q_4$ this lack of critical points of the classical mirror was already observed in \cite{EHX}. It was suggested there to solve it using a partial compactification and this was carried out for the first time, albeit in an ad hoc fashion. This was also a motivation for introducing the Lie-theoretic mirrors $(\RR,\F)$ in \cite{rietsch}. In the odd quadrics case Proposition~\ref{prop:Giv2W} is proved using a combination of results from \cite{GS} and \cite{PR2}. In the even quadrics case we prove it in the present paper. 

The first part of Proposition~\ref{p:critical} is an immediate consequence of analogous result for $(\RR,\F)$ from \cite{rietsch} together with Theorem~\ref{t:odd-quadric} and Theorem~\ref{t:even}, respectively. The second part comes from a direct calculation, see Propositions \ref{p:critical-odd} and \ref{p:critical:even}. 

\subsubsection*{Comparison with $(\TT,\LL)$} 

For odd quadrics $\Q_{2m-1}$ it was proved in \cite{PR2} that after a change of variables, $\TT$ gets identified 
with a particular torus inside $\XXcan$. This embedding identifies the two superpotentials $\LL$ and $\W$. We recall this result in Section \ref{comparison:odd}.

For even quadrics $\Q_{2m-2}$, the situation is more complicated. We consider the complement of a particular hyperplane section in $\TT$ for which we construct an embedding into $\XXcan$ such that $\W$ pulls back to $\LL$ and show that this embedding cannot be extended. Moreover we observe that the image of the embedding is precisely the embedded Givental mirror manifold inside $\XXcan$. Therefore the Givental mirror manifold is in a sense the intersection of the mirrors  $\TT$ and $\XXcan$. These results are contained in Section \ref{comparison:even}.

\subsubsection*{Comparison with the quiver mirror}

The quiver mirror $(\XXlus,\HH)$ is obtained from the Lie-theoretic mirror $(\RR,\F)$, and hence from the canonical mirror $(\XXcan,\W)$, by restricting it to a torus (see Propositions \ref{p:odd-quadric} and \ref{p:W1}).

\subsection{The mirror theorem for $A$-model and $B$-model $D$-modules}

\subsection*{}

Recall that the Dubrovin connection for $Q_N$ gives rise to a module $M_A$ over the ring of differential operators $D=\C[\hbar^{\pm 1}, q^{\pm 1}]\langle \partial_{\hbar},\partial_q\rangle$, see \eqref{e:MA}. On the $B$-side we obtain a  $D$-module $M_B$ by considering a Gauss-Manin system associated to the mirror $(\XXcan,\W)$, see Definition~\ref{d:GM}.  
For odd-dimensional quadrics it is already known that there is an isomorphism between $M_A$ and $M_B$. This follows from \cite[Section 4 \& Appendix A]{GS} together with the comparison result in \cite{PR2}. The isomorphism takes a particularly natural form in the canonical coordinates, as recalled in Theorem~\ref{t:connections-odd}. 

For even dimensional quadrics we construct in Section \ref{sec:connections} an explicit isomorphism from the $A$-model $D$-module $M_A$  to a natural submodule of the $B$-model $D$-module $M_B$, see Theorem~\ref{t:connections-even}. We conjecture that this submodule is in fact all of $M_B$, so that $M_A$ and $ M_B$ are isomorphic. Here our canonical mirror $(\XXcan,\W)$ takes place on a dual quadric. We note that there is a non-trivial cluster algebra structure on the coordinate ring of $\XXcan$, which plays an important role in our proof of the isomorphism.

\subsection{Applications}\label{s:applications}


\subsection*{}

In Section \ref{s:aseries}, we turn to the problem of constructing flat sections $S:\C^*_{\hbar}\times\C^*_q\to H^*(X,\C)$ for a dual version of the Dubrovin connection of the quadric $Q_N$, using the $B$-model. 
Namely we are interested in solutions to the  partial differential equation  
\begin{equation}\label{e:dualDubrovinPDE}
\begin{array}{ccl}	q\frac{\partial S}{\partial q} &=& \ \  \frac{1}{\hbar} \sigma_{1}\star_{q}S, \\
	 \hbar\frac{\partial S}{\partial\hbar}&=& - \frac{1}{\hbar} c_1(TX)\star_{q}S - \operatorname{Gr} (S).
\end{array}
\end{equation}

First we observe that one can write coefficients of flat sections from the $B$-model as oscillating integrals using $(\XXcan,\W)$. This goes as follows.
Consider any critical point $p$ of $\W$. By a procedure outlined by Givental in the setting of full flag varieties in \cite[Section~2]{Givental:QToda}, there should be an associated non-compact, middle-dimensional cycle $\Gamma_p$ in $\XXcan$ for which $\Re(\frac 1{\hbar}\W)\to -\infty$ rapidly in any unbounded direction of $\Gamma_p$ (here we suppress the dependence on $\hbar $ and $q$ in the notation for simplicity). Then, as in \cite[Section 4.2]{MR}, the integrals $\int_{\Gamma_p} e^{\frac{1}\hbar\W}p_i\omega$ locally determine coefficients of a section $S_{\Gamma_p}$. This section is given by the formula
\[
S_{\Gamma_p}=\frac{1}{(2\pi i)^N}\sum_{i=0}^{N}\left(	\int_{\Gamma_p} e^{\frac{1}\hbar\W}p_i\omegaCan\right) \sigma_{N-i}
\]
in the odd quadric case, and by a similar formula in the even quadric case. The local section $S_{\Gamma_p}$
 is a solution to \eqref{e:dualDubrovinPDE} as a consequence of Theorems~\ref{t:connections-odd} and \ref{t:connections-even}.
   With an appropriate partial compactification these cycles should have an interpretation in terms of  Lefschetz thimbles, compare \cite{Seidel:book}.

If we replace the cycle $\Gamma_p$ with a compact torus $(S^1)^N$ we obtain a global holomorphic flat section of the dual Dubrovin connection whose coefficients are given by residue integrals. In Section \ref{s:aseries} we construct this solution explicitly by expanding it as a power series, using the quiver mirror $(\XXlus,\HH)$ to express the integrals in coordinates. Moreover we verify the resulting formula in a different way on the $A$-side.





\section{Landau-Ginzburg models for odd quadrics}

The quadrics  are cominuscule homogeneous spaces (for the Spin groups). Therefore, in addition to the Givental approach \cite{Givental:Toric} for constructing LG~models, there is another LG~model for each quadric on an affine variety (generally larger than a torus), which was defined by the second-named author using a Lie-theoretic construction \cite{rietsch}. Namely for any projective homogeneous space $X=G/P$ of a simple complex algebraic group, \cite{rietsch} constructed a conjectural LG~model, which is a regular function on an affine subvariety of the Langlands dual group. We call it the \emph{Lie-theoretic LG model}. It was shown in \cite{rietsch} that this LG~model recovers the Peterson variety presentation \cite{peterson} of the quantum cohomology of $X=G/P$. It therefore defines an LG~model whose  Jacobi ring has the correct dimension. In this section we will rewrite the Lie-theoretic LG~model in terms of natural projective coordinates on $\P(H^*(\Q_N,\C)^*)$. We call the resulting LG model the \emph{canonical LG model} of $\Q_N$.

Note that for odd-dimensional quadrics $\Q_{2m-1}$ a recent paper \cite{GS} of Gorbounov and Smirnov constructed directly a partial compactification of the Givental mirrors, without making use of \cite{rietsch}. 

\subsection{The canonical LG model for $\Q_{2m-1}$}

LG models for odd-dimensional quadrics with the expected number of critical points have been constructed in \cite{rietsch} (where they appear as a special case), \cite{GS}, and finally \cite{PR2}. Here we recall the main results from the paper \cite{PR2}, which contains the formulation for the LG model which we will adopt.

In this section our $A$-model variety $X=X_N = X_{2m-1}$ is the quadric $Q_N = Q_{2m-1}$. Recall that an odd-dimensional quadric has  one-dimensional cohomology groups in even degrees spanned by Schubert classes $\sigma_i \in H^{2i}(\Q_{2m-1},\C)$ for $0\le i\le 2m-1$, and no other cohomology. To construct its canonical mirror first consider the projective space $\X=\X_{2m-1}=\P^{2m-1}$ with homogeneous coordinates $(\p_0:\p_1:\dots:\p_{2m-1})$ in one-to-one correspondence with these Schubert classes $\sigma_i$. Inside $\X$ we have the open affine subvariety $\XXcan \subset \P^{2m-1}$ defined by:
\begin{equation}\label{e:domain}
	\XXcan = \XX_{2m-1}:= \X \setminus D,
\end{equation}
where $D:=D_0+D_1+\dotsc + D_{m-1}+ D_m$, the divisors $D_i$ being given by
\begin{align*}
 	&D_0:= \left\{ \p_0 = 0 \right\}, \\
 	&D_{\ell}:= \left\{ \sum_{k=0}^{\ell} (-1)^k \p_{\ell-k} \p_{2m-1-\ell+k} = 0 \right\} \text{for $1\leq \ell \leq m-1$,}\\
 	&D_m:= \left\{ \p_{2m-1} = 0 \right\}.
\end{align*}
The divisor $D$ is an anticanonical divisor. Indeed, the index of $\X=\P^{2m-1}$ is $2m$. As a result, there is a unique up to scalar $(2m-1)$-form $\omegaCan$ which is regular on $\XXcan$ and has logarithmic poles on $D$. For all $1 \leq j \leq m-1$, take $r_j \in \{ \p_j, \p_{2m-1-j} \}$. Setting $\p_0=1$, the restriction of $\omegaCan$ to the torus $\{ r_j \neq 0 \mid 1 \leq j \leq m-1 \}$ inside $\XXcan$ is given by
\begin{equation}\label{e:omegaCanOdd}
	\omegaCan = \frac{\bigwedge_{ 1 \leq j \leq m-1} \rmd r_j \wedge \bigwedge_{ 1 \leq \ell \leq m-1} \rmd\delta_\ell \wedge \rmd\p_{2m-1}}{\delta_1 \dots \delta_{m-1} \p_{2m-1}}.
\end{equation}


We have:

\begin{theorem}[{\cite[Theorem 1]{PR2}}]\label{t:odd-quadric}
	The Lie-theoretic LG model $\F:\RR\to \C$ from \cite{rietsch} for $X=\Q_{2m-1}$ is isomorphic to the canonical LG model $\W:\XX_{2m-1}\to \C$ defined by 
    \begin{equation}\label{eq:W2}
    	\W = \frac{\p_{1}}{\p_0} + \sum_{\ell=1}^{m-1} \frac{\p_{\ell+1} \p_{2m-1-\ell}}{\delta_\ell} + q\frac{\p_{1}}{\p_{2m-1}},
    \end{equation}
	where $\delta_{\ell}$ is given by \eqref{eq:delta} with $N=2m-1$.
\end{theorem}

We also have another expression for the superpotential:

\begin{prop}[{\cite[Proposition 8]{PR2}}]\label{p:odd-quadric}
	For $X=Q_{2m-1}$ and $\W$ as above, there is a torus $\XXlus:=(\C^*)^{2m-1}\hookrightarrow \XXcan$ to which $\W$ pulls back giving the Laurent polynomial expression
	\begin{equation*}
   		\HH = a_1 + \dots + a_{m-1} + c +  b_{m-1} + \dots + b_1 + q \frac{a_1+b_1}{a_1 \dots a_{m-1} c  b_{m-1} \dots b_1}. 
	\end{equation*} 
\end{prop}

We call the Laurent polynomial LG model $(\XXlus,\W)$ from Proposition \ref{p:odd-quadric} the \emph{quiver mirror}. The reason for this denomination will be made clear in Section \ref{sec:quiver}.

\subsection{Comparison with the Givental and Laurent polynomial mirrors for odd quadrics}\label{comparison:odd}

Let us recall the Laurent polynomial LG model of $\Q_{2m-1}$ from Equation \eqref{e:toricLG}
\[
	\LL = z_1 + \dots + z_{2m-2} + \frac{(z_{2m-1}+q)^2}{z_1 z_2 \dots z_{2m-1}},
\]
defined over the torus
\[
	\TT:=\left\{ (z_1,\dots,z_{2m-1}) \mid z_i \neq 0\quad \forall\; i\right\},
\] 
and the Givental LG model from Equation \eqref{e:giventalLG}
\[
	\G = \nu_1 + \dots + \nu_{2m-1},
\]
defined over the affine variety 
\[
	\VV=\left\{ (\nu_1,\dots,\nu_{2m+1}) \mid \nu_i \neq 0\; \forall\; i,\,\  \prod_{i=1}^{2m+1}\nu_i=q,\ \ \nu_{2m}+\nu_{2m+1}=1\right\}.
\]
These two LG models are related by a birational change of coordinates analogous to that of \cite[Rmk.~19]{przyjalkowski}, namely
\begin{align*}
	z_i = \begin{cases}
			\nu_{i+1} & \text{for $1 \leq i \leq 2m-2$;} \\
			q\frac{\nu_{2m}}{\nu_{2m+1}} & \text{for $i=2m-1$;} 
		  \end{cases}
\end{align*}
and conversely
\begin{align*}
	\nu_i = \begin{cases}
			\frac{(z_{2m-1}+q)^2}{z_1 \dots z_{2m-1}} & \text{for $i=1$;} \\	
			z_{i-1} & \text{for $2 \leq i \leq 2m-1$;} \\
			\frac{z_{2m-1}}{z_{2m-1}+q} & \text{for $i=2m$;} \\
			\frac{q}{z_{2m-1}+q} & \text{for $i=2m+1$.}
		  \end{cases}
\end{align*}
This change of variables defines an isomorphism
\[
	\TT \setminus \{ z_{2m-1}+q=0 \} \cong \VV
\]
which identifies the superpotentials $\LL$ and $\G$.

Let us now compare these two LG models with ours. Consider the change of coordinates
\begin{align*}
	z_i = \begin{cases}
			\frac{\p_i}{\p_{i-1}} & \text{for $1 \leq i \leq m-1$;} \\
			\frac{\p_{2m-1-i}\delta_{2m-3-i}}{\p_{2m-2-i}\delta_{2m-2-i}} & \text{for $m \leq i \leq 2m-3$;} \\
			q\frac{\p_1}{\p_{2m-1}} & \text{for $i=2m-2$;} \\
			q\frac{\delta_{m-2}}{\delta_{m-1}} & \text{for $i=2m-1$.} 
		  \end{cases}
\end{align*}
It is well-defined on the cluster torus $\{ \p_i \neq 0 \; |\, \forall\, 1 \leq i \leq m-1 \}$ inside $\XXcan$. Moreover, an easy calculation shows that it transforms the canonical LG model \eqref{eq:W2} into the Laurent polynomial LG model \eqref{e:toricLG} for odd quadrics. 

Indeed, using this change of variables we see that $z_1 \dots z_{2m-1}$ maps to 
\[
	\frac{\p_{m-1}}{\p_0}\cdot \frac{\p_{m-1}\delta_0}{\delta_{m-2}\p_1}\cdot q\frac{\p_1}{\p_{2m-1}} \cdot q\frac{\delta_{m-2}}{\delta_{m-1}} = q^2\frac{(\p_{m-1})^2}{\delta_{m-1}}.
\]
Moreover $(z_{2m-1}+q)^2$ maps to $\left(q\frac{\p_{m-1}\p_m}{\delta_{m-1}}\right)^2$ since $\delta_{m-1}+\delta_{m-2}=\p_{m-1}\p_m$. It follows that $\frac{(z_{2m-1}+q)^2}{z_1 \dots z_{2m-1}}$ maps to $\frac{\p_m^2}{\delta_{m-1}}$. We also see that for $2 \leq j \leq m-1$, $z_j+z_{2m-1-j}$ maps to $\frac{\p_j\p_{2m-j}}{\delta_{j-1}}$ since $\delta_{j-1}+\delta_{j-2} = \p_{j-1} \p_{2m-j}$. Hence via this change of variables, the Laurent polynomial superpotential $\LL$ maps to
\[
	\frac{\p_1}{\p_0} + \sum_{j=2}^{m-1} \frac{\p_j\p_{2m-j}}{\delta_{j-1}} + q \frac{\p_1}{\p_{2m-1}} +\frac{\p_m^2}{\delta_{m-1}},
\]
which is precisely the expression of $\W$.

Note that this change of coordinates between $(\TT,\LL)$ and $(\XXcan,\W)$ may also be obtained by combining the isomorphism between  \eqref{eq:W2} and the Gorbounov-Smirnov mirror from \cite[Section 6]{PR2}, with the comparison between the Gorbounov-Smirnov mirror and the Laurent polynomial mirror (there called the Hori-Vafa mirror) in \cite{GS}.

Combining both changes of coordinates, we obtain an embedding of the Givental mirror variety $\VV \hookrightarrow \XXcan$, corresponding to the change of coordinates
\begin{align*}
	\nu_i = \begin{cases}
			\frac{\p_m^2}{\delta_{m-1}} & \text{for $i=1$;} \\
			\frac{\p_{i-1}}{\p_{i-2}} & \text{for $2 \leq i \leq m$;} \\
			\frac{\p_{2m-i}\delta_{2m-2-i}}{\p_{2m-1-i}\delta_{2m-1-i}} & \text{for $m+1 \leq i \leq 2m-2$;} \\
			q\frac{\p_1}{\p_{2m-1}} & \text{for $i=2m-1$;} \\
			\frac{\delta_{m-2}}{\p_{m-1}\p_m} & \text{for $i=2m$;} \\
			\frac{\delta_{m-1}}{\p_{m-1}\p_m} & \text{for $i=2m+1$.} 
		  \end{cases}
\end{align*}
The embedding identifies $\VV$ with the intersection of cluster tori $\{ \p_i \neq 0 \; |\, \forall\, 1 \leq i \leq m \}$ in $\XX$, the superpotential $\G$ with $\W$, and the form $\omegaGiv$ with $\omegaCan$. This proves Proposition \ref{prop:Giv2W} from the introduction in the case of odd quadrics.

\subsection{The critical points of the canonical mirror}

Since the canonical mirror $(\XXcan,\W)$ is isomorphic to the Lie-theoretic mirror $(\RR,\F)$, it follows from \cite{rietsch} that $\W$ has the `correct' number of critical points on $\XXcan$, that is, $\dim H^*(\Q_{2m-1},\C)=2m$. Here we give explicit expression for the critical points, and compare with the critical points of the classical mirrors $(\VV,\G)$ and $(\TT,\LL)$.

\begin{prop}
\label{p:critical-odd}
	The critical points of the superpotential $\W$ on $\XXcan$ are given by
	\begin{align*}
		\p_j = 	\begin{cases}
					\zeta^j & \text{if $1 \leq j \leq m-1$ ;} \\
					\frac{1}{2}\zeta^j & \text{if $m \leq j \leq 2m-2$ ;} \\
					q & \text{if $j=2m-1$},
				\end{cases}
	\end{align*}
	where $\zeta$ is a primitive $(2m-1)$-st root of $4q$. The associated critical value is $(2m-1)\zeta$.
	Moreover there is an extra critical point given by $\p_1 = \dots = \p_{2m-2} = 0$, $\p_{2m-1}=-q$ with corresponding critical value $0$. This critical point does not belong to $\TT$, $\VV$ or $\XXlus$.
\end{prop}

\begin{proof}
Setting $\p_0=1$ we get the following relations at a critical point of $\W$:
\begin{align}
	\frac{\partial \W}{\partial \p_1} &= 1 + \left(\sum_{\ell=1}^{m-1} (-1)^\ell \frac{\p_{\ell+1}\p_{2m-1-\ell}}{\delta_\ell^2}\right) \p_{2m-2} + \frac{q}{\p_{2m-1}} = 0 \label{e:1}\\
	\frac{\partial \W}{\partial  \p_j} &= \frac{\p_{2m-j}}{\delta_{j-1}} + \left( \sum_{\ell=j}^{m-1} (-1)^{\ell+1-j} \frac{\p_{\ell+1}\p_{2m-1-\ell}}{\delta_\ell^2} \right) \p_{2m-1-j} = 0\; (2 \leq j \leq m-1) \label{e:j} \\
	\frac{\partial \W}{\partial  \p_m} &= \frac{\p_m(\p_{m-1}\p_m-2\delta_{m-2})}{\delta_{m-1}^2}=0 \label{e:m}\\
	\frac{\partial \W}{\partial  \p_j} &= -\frac{\p_{2m-j}\delta_{2m-2-j}}{\delta_{2m-1-j}^2} + \left( \sum_{\ell = 2m-j}^{m-1} (-1)^{\ell+j-2m} \frac{\p_{\ell+1}\p_{2m-1-\ell}}{\delta_\ell^2} \right)\p_{2m-1-j}=0\; (m+1 \leq j \leq 2m-2) \label{e:k} \\
	\frac{\partial \W}{\partial  \p_{2m-1}} &= \sum_{\ell=1}^{m-1} (-1)^{\ell-1} \frac{\p_{\ell+1}\p_{2m-1-\ell}}{\delta_\ell^2}-q\frac{\p_1}{\p_{2m-1}^2}=0 \label{e:2m-1}.
\end{align}

From Equation \eqref{e:m} it follows that we have two possibilities, i.e. $\p_m=0$, or $\p_{m-1}\p_m=2\delta_{m-2}$. If $\p_m=0$, using \eqref{e:j} for $j=m-1,m-2,\dots,2$ shows that $\p_m=\p_{m+1}=\p_{2m-2}=0$. Then \eqref{e:1} implies $\p_{2m-1}=-q$. Using \eqref{e:k} for $j=m+1,m+2,\dots,2m-2$ shows that $\p_{m-1}=\p_{m-2}=\dots=\p_2=0$. Finally \eqref{e:2m-1} implies $\p_1=0$. At the corresponding critical point $(0,\dots,0,-q)$, the value of $\W$ (the critical value) is clearly $0$.

Let us now assume $\p_m \neq 0$ and $\p_{m-1}\p_m=2\delta_{m-2}$, so that $\delta_{m-1}=\delta_{m-2}$. Combining Equations \eqref{e:j} for $j=m-1$ and \eqref{e:k} for $j=m+1$, we obtain $\p_{m-2}\p_{m+1}=2 \delta_{m-3}$, hence $\delta_{m-2}=\delta_{m-3}$. Iteratively, we obtain
\begin{align}
	&\delta_{m-1}=\delta_{m-2}=\dots=\delta_0 ; \label{e:deltaEq} \\
	&\p_j \p_{2m-1-j} = 2 \delta_{j-1} \quad \forall\; 1\leq j \leq m-1.
\end{align}

Combining Equations \eqref{e:1} and \eqref{e:2m-1} with the identity \eqref{e:deltaEq}, we get that
\begin{equation}
	\p_{2m-1}=q,
\end{equation}
hence all the $\delta_j$ are equal to $q$.

Now Equations \eqref{e:j} for $j=m-1$ and \eqref{e:deltaEq} imply $\p_{m-1}\p_{m+1}=2\p_m^2$. Then Equation \eqref{e:j} for $j=m-2$ and \eqref{e:deltaEq} imply that $\p_{m-2}\p_{m+2}=2(\p_{m-1}\p_{m+1}-\p_m^2)$. Inductively for $j=m-3,\dots,2$ we obtain 
\begin{equation}\label{e:star}
	\sum_{\ell=j}^m (-1)^{\ell-j} \p_\ell \p_{2m-\ell} = \p_m^2
\end{equation}
for all $2 \leq j \leq m$. Then \eqref{e:1} implies that $\p_m^2=q\p_1$.

Finally, \eqref{e:k} and \eqref{e:star} imply $\p_j=\frac{1}{2}\p_1^j$ for $m \leq j \leq 2m-2$, while \eqref{e:j} and \eqref{e:star} imply $\p_j=\p_1^j$ for $1 \leq j \leq m-1$. Then $\p_m=\frac{1}{2}\p_1^m$ together with $\p_m^2=q\p_1$ implies that $\p_1^{2m-1}=4q$, which concludes the proof.
\end{proof}

\section{Landau-Ginzburg models for even quadrics}\label{s:LGeven}

We view the quadric $X=X_{2m-2}:=\Q_{2m-2}$ of dimension $2m-2$ as a homogeneous space for the Spin group $\mathrm{Spin}_{2m}(\C)$. In this section we will introduce a \emph{canonical LG model} for $X_{2m-2}$ which will be defined on an open subvariety of a dual quadric $\X_{2m-2}=P\backslash \PSO_{2m}(\C)$, see Section~\ref{s:dualquadric}. Note that the projective special orthogonal group $\PSO_{2m}(\C)$ is the Langlands dual group to $\mathrm{Spin}_{2m}(\C)$, and both groups have the same Dynkin diagram, namely the Dynkin diagram of type~$D_m$. The main result of this section, Proposition~\ref{p:iso-even}, shows that the new LG-model is isomorphic to one defined earlier \cite{rietsch} on a Richardson variety  $\RR$ inside the full flag variety of $\PSO_{2m}(\C)$. 

Note that in the following we will denote the group $\PSO_{2m}(\C)$  by $G$, since this is the group we will primarily be working with. Then  the  $A$-model symmetry group is  $G^\vee=\mathrm{Spin}_{2m}(\C)$, and we have $X_{2m-2}=G^\vee/P^\vee$, where $P^\vee$ is the parabolic subgroup  associated to the first node of the Dynkin diagram of type~$D_m$.

\begin{center}
\begin{tikzpicture}
	\tikzstyle{every node}=[draw,circle,fill=white,minimum size=4pt,
                            inner sep=0pt,label distance=2pt]
                            
	\node[draw,fill=black,circle,minimum size=4pt,label=below:{1}] (1) at (0,0) {};
	\node[label=below:{2}] (2) at (1,0) {};
	\node[label=below:{3}] (3) at (2,0) {};
	\node (4) at (3,0) {};
	\node[label=below:{$m-2$}] (5) at (4,0) {};
	\node[label=right:{$m$}] (m) at (5,0.5) {};
	\node[label=right:{$m-1$}] (m-1) at (5,-0.5) {};
	
	\draw
	(1) -- (2)
	(2) -- (3)
	(4) -- (5)
	(5) -- (m)
	(5) -- (m-1);
	
	\draw[loosely dotted]
	(3) -- (4);
\end{tikzpicture}
\end{center}

\subsection{Notations and definitions}\label{s:notationeven}

Let $V=\C^{2m}$ with fixed quadratic form
\[
	Q=	\begin{pmatrix}
    		&  &  & & 1\\
    		&  &  & -1 & \\
    		& & \Ddots & \\
    		& -1 &  &  &\\
    		1 & & & &
   		\end{pmatrix}.
\] 
In other words $Q(v_{i},v_{j})=(-1)^{\max(i,j)}\delta_{i+j,2m+1}$ where $\{v_i\}$ is the standard basis of $\C^{2m}$. For $G=\PSO(V,Q)=\PSO(V)$ we fix Chevalley generators $(e_i)_{1 \leq i \leq m}$ and $(f_i)_{1 \leq i \leq m}$. To be explicit we embed $\mathfrak{so}(V,Q)$ into $\mathfrak{gl}(V)$ and set
\begin{align*}
	e_i = \begin{cases}
			E_{i,i+1}+E_{2m-i,2m-i+1}  & \text{if $1 \leq i \leq m-1$,} \\
			E_{m-1,m+1}+E_{m,m+2} & \text{if $i=m$,}
		  \end{cases}
\end{align*}
and $f_i:=e_i^T$, the transpose matrix, for every $i=1,\dotsc, m$. Here $E_{i,j}=(\delta_{i,k}\delta_{l,j})_{k,l}$ is the standard basis of $\mathfrak{gl}(V)$. For elements of the group $\PSO(V)$, we will take matrices to represent their equivalence classes. We have Borel subgroups $B_+=T U_+$ and $B_-=T U_-$ consisting of upper-triangular and lower-triangular matrices in $\PSO(V)$, respectively. Here $U_+$ and $U_-$ are the unipotent radicals of $B_+$ and $B_-$, respectively, and $T$ is the maximal torus of $\PSO(V)$,  consisting of diagonal matrices  $(d_{ij})$ with non-zero entries $d_{i,i}=d_{2m-i+1,2m-i+1}^{-1}$. We let $X(T) = \Hom(T,\C^*)$, $R \subset X(T)$ the set of roots, and $R^+$ the positive roots.  We denote the set of simple roots by 
$\Pi= \{\alpha_i \ \vert \ 1 \leq i \leq m\} \subset R^+ \subset R \subset X(T)$, and the set of fundamental weights (which is the dual basis in $X(T)$) by $\{\omega_i \ \vert \ 1 \leq i \leq m\} \subset X(T) \otimes_{\Z} \R$.
 
The parabolic subgroup $P$ of $\PSO(V)$ we are interested in is the one whose Lie algebra $\mathfrak p$ is generated by all of the $e_i$ together with $f_2,\dotsc, f_{m}$, leaving out $f_1$. Let $x_i(a):=\exp(a e_i)$ and $y_i(a):=\exp(a f_i)$. The Weyl group $W$ of $\PSO(V)$ is generated by simple reflections $s_i$ for which we choose representatives
\begin{equation}\label{e:si}
	\dot s_i=y_i(-1)x_i(1)y_i(-1).
\end{equation}
We let $W_P$ denote the parabolic subgroup of the Weyl group $W$, namely $W_P=\langle s_2,\dotsc, s_m\rangle$. The length of a Weyl group element $w$ is denoted by $\ell(w)$. The longest element in $W_P$ is denoted by $w_P$. We also let $w_0$ be the longest element in $W$. Next $W^P$ is defined to be the set of minimal length coset representatives for $W/W_P$. The minimal length coset representative for $w_0$ is denoted by $w^P$.

We introduce the following notation for the elements of $W^P$. Namely, $W^P=\{e, w_1,\dotsc, w_{m-1},w_{m-1}', w_m,w_{m+1},\dotsc w_{2m-2}\}$, where
\begin{align*}
  	w_{k} = 	\begin{cases}
         			s_k s_{k-1} \dots s_1 & \text{if $1\le k \leq m-2$,} \\
         			s_{m-1} s_{m-2} \dots s_1 & \text{if $k = m-1$,} \\
         			s_{m} s_{m-1} s_{m-2} \dots s_1 & \text{if $k = m$,} \\
         			s_{2m-1-k} \dots s_{m-2} s_m s_{m-1} s_{m-2} \dots s_1 & \text{if $m+1 \leq k \leq 2m-2$.}
        		\end{cases}
\end{align*}
and $w_{m-1}' = s_{m} s_{m-2} \dots s_1$. 
 
For any $w\in W$ let $\dot{w}$ denote the representative of $w$ in $G$ obtained by setting $\dot w=\dot s_{i_1}\cdots \dot s_{i_r}$, where $w= s_{i_1}\cdots s_{i_r}$ is a reduced expression and $\dot s_i$ is as in \eqref{e:si}. Each $\dot w_k\in \PSO(V)$ can be represented by a matrix $[w_k]\in \SO(V)$ such that 
\begin{equation}\label{e:basisV} 
	[w_k]\cdot v_{2m}=\begin{cases}v_{2m-k} & 1\le k < m-1,\\
v_{2m-k-1} & m-1< k\le 2m-2, 
 \end{cases}
\end{equation}
and $ [w_{m-1}']\cdot v_{2m}=v_{m}$ and $ [w_{m-1}]\cdot v_{2m}=v_{m+1}$.

\subsection{The dual quadric and its Pl\"ucker coordinates}\label{s:dualquadric}

Consider the homogeneous space $\X_{2m-2}= P\backslash \PSO(V)$. It is canonically identified with the isotropic Grassmannian of lines in $V^*$, when this Grassmannian is viewed as a homogeneous space via the action of $\PSO(V)$ from the right. Moreover the isotropic Grassmannian of lines is also a $(2m-2)$-dimensional quadric $\X_{2m-2} =:\check \Q_{2m-2}$, now in $\mathbb P(V^*)$. So in this  case, the varieties $X$ and $\X$ are (non-canonically) isomorphic. The reason for this isomorphism of varieties is that the group $G^\vee$ is of simply-laced type. However Lie-theoretically we still think of $X_{2m-2}$ and $\X_{2m-2}$ as being very different homogeneous spaces, with $X_{2m-2}=\Spin_{2m}(\C)/P^\vee$ and $\X_{2m-2}=P\backslash \PSO_{2m}(\C)$.

\begin{defn}[Pl\"ucker coordinates]\label{def:plucker} 
	The Pl\"ucker coordinates for $\X_{2m-2}=P\backslash \PSO(V)$ are the homogeneous coordinates coming from the embedding of $\X_{2m-2}$ into $\mathbb P(V^*)$ as the (right) $G$-orbit of the line $\C v_{2m}^*$:
	\[
		\X_{2m-2}=P\backslash \PSO(V)\to \mathbb P(V^*):  Pg\mapsto (\C v_{2m}^*)\cdot g.
	\]
	We think of the Pl\"ucker coordinates as corresponding to the elements of $W^P$. Let $v_{\omega_i}^-$ (respectively $v_{\omega_i}^+$) denote lowest and highest weight vectors in the highest weight representation $V_{\omega_i}$. Then the Pl\"ucker coordinates may be defined by:
	\begin{align*}
		 \p_0(g) & = \langle v_{2m}^*\cdot [g],  v_{2m}\rangle, \\
 		 \p_k(g) & = \langle v_{2m}^*\cdot [g], [w_k] \cdot v_{2m} \rangle \text{ for }1 \leq k \leq 2m-2,\\
 		 \p_{m-1}'(g) &= \langle v_{2m}^*\cdot [g], [w_{m-1}'] \cdot v_{2m} \rangle,
	\end{align*}
 	where $[g]\in \SO(V)$ is any fixed matrix representing $g\in \PSO(V)$. The homogeneous coordinates of $Pg$ are then given by 
 	\[
		(\p_0(g):\dotsc:\p_{m-2}(g):\p_{m-1}(g):\p_{m-1}'(g):\p_m(g):\dotsc\ :\p_{2m-2}(g)).
	\]
	These are simply the bottom row entries of $[g]$ read from right to left, keeping in mind \eqref{e:basisV}. 
\end{defn}
 
We may now write down the equation of the quadric $\X_{2m-2}$ in terms of Pl\"ucker coordinates:
\begin{equation}\label{e:equationQuadric}
	\p_{m-1} \p_{m-1}' - \p_{m-2} \p_m + \p_{m-3} \p_{m+1} - \dots + (-1)^{m-1} \p_0 \p_{2m-2} = 0.
\end{equation}
 
We note that as in the case of the odd quadric these Pl\"ucker coordinates are to be thought of as $B$-model incarnations of the Schubert classes of $Q_{2m-2}$. Namely, recall that $H^*(Q_{2m-2},\C)$ has a Schubert basis $\{\sigma_w\}$ indexed by $W^P$. We will use the notation $\sigma_i=\sigma_{w_i}$, $\sigma_{m-1}'=\sigma_{w_{m-1}'}$, and $\sigma_0=\sigma_e$, where the $w_i$ are defined in Section~\ref{s:notationeven}. As a special case of the geometric Satake correspondence \cite{lusztig,Gin:GS,MV} we have that the (defining) projective representation $V$ of $PSO_{2m}(V)$ is identified with the cohomology of $Q_{2m-2}$,
\[
	V=H^*(Q_{2m-2},\C),
\]
and the standard basis $v_i$ agrees with the Schubert basis via $v_{2m}=\sigma_0$ and
\begin{equation}\label{e:SchubertGS}
	[w_{i}]\cdot v_{2m} = \sigma_i, \quad [w_{m-1}']\cdot v_{2m} = \sigma_{m-1}'.
\end{equation}
The Schubert classes $\sigma_w$ are in this way naturally identified with the  Pl\"ucker coordinates.

\subsection{The superpotential for $Q_{2m-2}$ on a dual quadric}

In this section we state our theorem describing a superpotential for $Q_{2m-2}$ in terms of Pl\"ucker coordinates on the dual quadric $\X_{2m-2}=\check Q_{2m-2}$. Consider
\begin{equation}\label{e:domain-even}
	\XXcan=\XX_{2m-2}:= \X \setminus D,
\end{equation}
where $D:=D_0+D_1+\dotsc + D_{m-2}+ D_{m-1} + D_{m-1}'$, the $D_i$ being given by
\begin{align*}
 	&D_0:= \left\{ \p_0 = 0 \right\}, \\
 	&D_\ell:= \left\{ \sum_{k=0}^\ell (-1)^k \p_{\ell-k} \p_{2m-2-\ell+k} = 0 \right\} \text{ for $1\leq \ell \leq m-3$,}\\
 	&D_{m-2}:= \left\{ \p_{2m-2} = 0 \right\}, \\
 	&D_{m-1}:= \left\{ \p_{m-1} = 0 \right\}, \\
 	&D_{m-1}':= \left\{ \p_{m-1}' = 0 \right\}.
\end{align*}
The divisor $D$ is an anticanonical divisor in $\X$ (see \cite[Lemma 5.4]{KLS}). 
For simplicity, we will define
\begin{equation}\label{eq:delta2}
	\delta_{\ell}  =  \sum_{k=0}^{\ell} (-1)^k \p_{\ell-k} \p_{N-\ell+k} \text{ for }1 \leq \ell \leq m-3.
\end{equation}
(For even quadrics, $N = 2m-2$.)

As in the odd case, we have a unique up to scalar $(2m-2)$-form $\omegaCan$ which is regular on $\XXcan$ and has logarithmic poles along $D$. For all $1 \leq j \leq m-2$, take $r_j \in \{ \p_j, \p_{2m-2-j} \}$. Setting $\p_0=1$, the restriction of $\omegaCan$ to the torus $\{ r_j \neq 0 \mid 1 \leq j \leq m-2 \}$ inside $\XXcan$ is given by
\begin{equation}\label{e:omegaCanEven}
	\omegaCan = \frac{\bigwedge_{ 1 \leq j \leq m-2} r_j \wedge \bigwedge_{ 1 \leq \ell \leq m-3} \delta_\ell \wedge \p_{2m-2} \wedge \p_{m-1} \wedge \p_{m-1}'}{\delta_1 \dots \delta_{m-3} \p_{2m-1} \p_{m-1} \p_{m-1}'}.
\end{equation}

Our first result is the following theorem.
\begin{theorem}\label{t:even}
	The Lie-theoretic LG model $(\RR,\F)$ for $\Q_{2m-2}=\mathrm{Spin}_{2m}/P^\vee$ from \cite{rietsch} is isomorphic to the canonical LG model $(\XX_{2m-2},\W)$, where $\W:\XX_{2m-2}\to \C$ is defined by 
	\begin{equation}\label{eq:Weven}
    	\W = \frac{\p_{1}}{\p_0} + \sum_{\ell=1}^{m-3} \frac{\p_{\ell+1} \p_{2m-2-\ell}}{\delta_\ell} + \frac{\p_m}{\p_{m-1}} + \frac{\p_m}{\p_{m-1}'} + q\frac{\p_{1}}{\p_{2m-2}}.
	\end{equation}
\end{theorem}
This isomorphism is defined in Section \ref{s:richardson-even}. Before we begin the proof we need to recall the definition of the Lie-theoretic LG model from \cite{rietsch}.

\subsection{The Lie-theoretic LG model for $Q_{2m-2}$}\label{s:richardson}

Following \cite{rietsch} consider the (open) Richardson variety $\RR:= R_{w_P,w_0} \subset G / B_-$, namely 
\[ 
	\RR:=R_{w_P,w_0} = (B_+\dot w_P B_-  \cap B_-\dot w_0 B_- )/ B_-.
\]
This Richardson variety $\RR$ is irreducible of dimension $2m-2$, and its closure is the Schubert variety $\overline{B_+\dot w_PB_-/B_-}$. Let $T^{W_P}$ be the $W_P$-fixed  part of the maximal torus $T$. Note that since we are in the setting of Section~\ref{s:notationeven} we have that $T^{W_P}\cong \C^*$ with isomorphism given by $\alpha_1$. The inverse isomorphism is $\omega_1^\vee:\C^*\to T^{W_P}$. We fix a $d\in T^{W_P}$. Then one can define
\begin{equation}\label{e:Zd}
	Z_d:=B_-\dot w_0\cap U_+ d\dot w_P U_- \subset G,  
\end{equation}
and the map 
\begin{equation}\label{e:piR}
	\pi_R: Z_d\to \RR: g \mapsto g B_-.
\end{equation}
is an isomorphism from $Z_d$ to the open Richardson variety \cite[Section 4.1]{rietsch}.

Let $q$ be the non-vanishing coordinate on the $1$-dimensional torus $T^{W_P}$ given by $\alpha_1:T^{W_P}\to\C^*$.  The mirror LG model is a regular function on $\RR$ depending also on $q$, and hence
a regular function on  $\RR\times T^{W_P}$. It is defined as follows~\cite{rietsch}:
\begin{equation}\label{e:F}
	\mathcal F: (u_1\dot w_P B_-,d)\ \mapsto\ g=u_1 d \dot w_P \bar u_2\in Z_d\ \mapsto\  \sum e_i^*(u_1)+\sum f_i^*(\bar u_2),  
\end{equation}
where $u_1 \in U_+, \bar u_2 \in U_-$, and where $\bar u_2$ is determined by $u_1$ and the property that 
$u_1 d \dot w_P \bar u_2\in Z_d$.

The corresponding map from $\RR$, when the coordinate $q$ is fixed, is denoted
\[
	\F: \RR \to \C: u_1\dot w_P B_- \mapsto \mathcal F(u_1\dot w_P B_-,\omega_1^\vee(q)).
\]

\begin{remark}
Note that if $g=u_1 d \dot w_P  \bar u_2 \in Z_d$, then we have a simple identity concerning
the Pl\"ucker coordinates:
\[
	(\p_0(g):\dotsc: \p_{2m-2}(g)) = (\p_0(\bar u_2):\dotsc:\p_{2m-2}(\bar u_2)).
\]
\end{remark}

The remainder of Section~\ref{s:LGeven} will be devoted to proving Theorem~\ref{t:even}, which now says that there is an isomorphism $\XX_{2m-2}\overset \sim\to \RR$ under which $\W$ is identified with~$\F$.

\subsection{Isomorphism between $\XXcan$ and $\RR$}
\label{s:richardson-even}

To prove Theorem \ref{t:even}, the first step is to construct an isomorphism between $\XX_{2m-2}$ and the open Richardson variety $\RR$. We define the following maps:
\begin{eqnarray*}
	\X=P\backslash G\ \overset{\pi_L}  \longleftarrow   &Z_d = B_-\dot w_0\cap U_+ d \dot w_P U_- &  \overset{\pi_R}{\longrightarrow}\  \RR,\\
 	P g\ \mapsfrom \ & g & \mapsto g B_-,
\end{eqnarray*}
given by taking left and right cosets, respectively. Note that $g$ is equal to $b_-\dot w_0$ in our previous notation and factorizes (a priori non-uniquely) as
\[
	g=u_1d\dot w_P\bar u_2.
\]
Moreover $\pi_R$ is an isomorphism, so we have $\pi:=\pi_L\circ\pi_R^{-1}:\RR\to \X_{2m-2}$. Our next goal is to prove:
\begin{prop}\label{p:iso-even}
	$\pi_L$ defines an isomorphism from $Z_d$ to $\XX_{2m-2}$. As a consequence, $\pi$ defines an isomorphism from $\RR$ to $\XX_{2m-2}$.
\end{prop}

Our proof uses a presentation of the coordinate ring of the unipotent cell 
\begin{equation}\label{e:unipotent-cell}
	U_-^P:= U_- \cap B_+ (\dot w^P)^{-1} B_+
\end{equation}
due to \cite{GLS}. The strategy of the proof of Proposition \ref{p:iso-even} is as follows.

\begin{itemize}
	\item  The first step is to show that the natural map $\pi_L : Z_d\to \X$ factorizes as $\phi\circ\theta$ where $\phi: U_-^P\to \X$ with $\phi(\bar u)=P\bar u$ and $\theta: Z_d\to U_-^P$ is an isomorphism which will be constructed in Lemma~\ref{l:theta}. 
	\item We then use the presentation of the coordinate ring of $U^P_-$ to show that the image of the map $\phi$ lands in $\XX_{2m-2}$ and not just $\X_{2m-2}$.  That is, the Pl\"ucker coordinates
$p_0$, $p_{2m-2}$, $p_{m-1}$, $p'_{m-1}$ and the functions $\delta_{\ell}$ (defined in \eqref{eq:delta}) do not vanish. Finally, we show that $\phi$ is an isomorphism from $U_-^P$ to $\XX_{2m-2}$. The main step is to find a pre-image for each of the functions generating $\C[U_-^P]$.  
\end{itemize}

\begin{lemma}\label{l:theta} 
	There exists an isomorphism $\theta:Z_d\to U^P_-$ such that for $b\dot w_0\in Z_d$, 
	\begin{equation}\label{e:Pu}
		Pb\dot w_0=P\bar u_2,
	\end{equation}
	where $\bar u_2:=\theta(b\dot w_0)$.
\end{lemma}

To prove Lemma~\ref{l:theta} we use an isomorphism  introduced by Berenstein and Zelevinsky in \cite{BZ} (and joint with Fomin in type $A$ \cite{BFZ}) which is sometimes called the BZ twist (or BFZ twist).  
\begin{theorem}\cite[Theorem~1.2]{BZ} 
	Let $y\in U_- \cap B_+ \dot w^{-1} B_+$. There exists a unique $x\in U_+ \cap B_- \dot w B_-$ such that $U_+\cap B_- \dot w y=\{x\}$. The resulting map $\tilde\eta_w:U_- \cap B_+ \dot w^{-1} B_+\to U_+ \cap B_- \dot w B_-$ sending $y$ to $x$ is an isomorphism. In particular we have an inverse isomorphism 
	\[
		\varepsilon_w:U_+ \cap B_- \dot w B_-\to U_- \cap B_+ \dot w^{-1} B_+.
	\]
\end{theorem}

\begin{remark} 
	We note that the original twist map of Berenstein and Zelevinsky  is an automorphism $\eta_w: U_+ \cap B_- \dot w B_-\to U_+ \cap B_- \dot w B_-$. Our map $\tilde{\eta}_w$ is related to $\eta_w$ by 
	\[
		\tilde\eta_w(y)=\eta_w(y^T),
	\] 
	where $y^T$ denotes the transpose of $y$.  We have
	\[
		\tilde\eta_w(y)=x \quad  \iff \quad B_-w y=B_-x.
	\]
	Here we may write $B_-w$ for $B_-\dot w$, as the coset doesn't depend on the representative of $w$.
\end{remark}

\begin{proof}[Proof of Lemma~\ref{l:theta}]
The idea is to consider the two birational maps
\[
	\begin{array}{clccccc}
		\Psi_1:&U^P_- &\to &P\backslash G\:, & \bar u_2&\mapsto & P\bar u_2, \\
	\pi_L: &Z_d &\to& P\backslash G\:, & b_- \dot w_0=u_1 d\dot w_P\bar u_2&\mapsto &P b_- \dot w_0, 
	\end{array}
\]
and to show that the composition 
\begin{equation}\label{e:rationalBZmap}
	\theta:=\Psi_1^{-1}\circ \pi_L: Z_d \to  U_-^P.
\end{equation}
is an isomorphism.  We construct a commutative triangle of maps as follows.
\begin{center}
	\begin{tikzpicture}
  		\matrix [matrix of math nodes,row sep=1cm, column sep=1cm,text height=1.5ex,text depth=0.25ex]
  		{
	  		& |(U)| U_-\dot w_0 \cap B_+ \dot w_P U_- & \\
	 		|(Z)| Z_d & & |(R)| U^P_- \\
		};
		\path[->]
			(Z) edge node[above] {$\mu$} (U)
		;
		\path[->]
			(U) edge node[above]  {$\xi$} (R) 
		;
		\path[->]
			(Z) edge node[below] {$\theta$} (R)
		;
	\end{tikzpicture}
\end{center}
Here  $\mu : Z_d \to U_- \dot w_0 \cap B_+ \dot w_P U_-$ is an isomorphism defined by $b_- \dot w_0 \mapsto [b_-]_0^{-1} b_- \dot w_0$, where $[b_-]_0$ is the torus part of $b_-$. The inverse isomorphism $\mu^{-1}$ is given by $b_+ \dot w_P u_- \mapsto d [b_+]_0^{-1} b_+  \dot w_P u_-$. Note that clearly $Pz=P\mu(z)$ for all $z\in Z_d$.

We now define a composition $\xi$ of isomorphisms as follows,
\begin{center}
	\begin{tikzpicture}
		\matrix [matrix of math nodes,row sep=1cm, column sep=1cm,text height=1.5ex,text depth=0.25ex]
  		{
	 		|(a)| U_-\dot w_0 \cap B_+ \dot w_P U_- & |(U)| U_+ \cap B_-  w^P  B_- & |(c)| U_- \cap B_+ (\dot w^P)^{-1} B_+, \\
		};
		\path[->]
			(a) edge node[above] {$\ell_{\dot w_0^{-1}}$} (U)
		;
		\path[->]
			(U) edge node[above]  {$\varepsilon_{w^P}$} (c) 
		;
	\end{tikzpicture}
\end{center}
where $\ell_{\dot w_0^{-1}}$ is the left multiplication by $\dot w_0^{-1}$ map. Hence we obtain an isomorphism 
\[
	\xi: U_-\dot w_0 \cap B_+ \dot w_P U_- {\longrightarrow} U^P_-.
\]
Suppose $u_-\dot w_0\in U_-\dot w_0 \cap B_+ \dot w_P U_-$. To prove the identity \eqref{e:Pu} it remains to check that $Pu_-\dot w_0=P\bar u_2$ where $\bar u_2=\xi(u_-\dot w_0)$. This follows from the defining property of $\varepsilon_{w^P}$. Namely if $u_-\dot w_0\in U^P_-\dot w_0$ then if $y=\varepsilon_{s^P}(\dot w_0^{-1}w \dot w_0)$, we have
\[
	B_- \dot w_0\inv u_- \dot w_0=B_- w^P\bar u_2=B_-w_0w_P\bar u_2.
\] 
Therefore $B_+ u_-\dot w_0=B_+ w_P \bar u_2$.
\end{proof}

For the second step of the proof of Proposition~\ref{p:iso-even} we use a result of \cite{GLS} to describe the coordinate ring of the unipotent cell $U_-^P$. In Lemma~\ref{l:iso} we then explicitly relate the coordinates on $U^P_-$ to the coordinates on $\XX_{2m-2}$, which are the Pl\"ucker coordinates from Definition~\ref{def:plucker}. In this way we show that the map
\[
	\phi: U^P_-\to \X,\quad \bar u_2\mapsto P\bar u_2
\]
restricts to an isomorphism onto its image, and that this image is $\XX_{2m-2}$.

We must first define the generalized minors involved in the presentation due to \cite{GLS}. Let $G^{sc}$ be the simply-connected covering group of $G=\PSO(V)$, with Borel subgroup $B^{sc}_-$ and unipotent radical $U^{sc}_-$ projecting to $B_-$ and $U_-$ in $G$. Here $G^{sc} = \Spin(V)$.  Since $U^{sc}_- \cong U_-$ via this projection, we may use representations of $G^{sc}$ to define generalized minors of elements of $U_-$. For $u\in U_-$ we denote by $u^{sc}$ its lift to $U_-^{sc}$, and similarly for elements of $U_+$.

Let $w \in W$ have reduced expression $w=s_{i_1} s_{i_2} \dots s_{i_r}$. Write 
\[
	\bar s_j = y_j^{sc}(1)x_j^{sc}(-1)y_j^{sc}(1)
\]
and $\bar w = \bar s_{i_1} \bar s_{i_2} \dots \bar s_{i_r}$.

\begin{defn}
	Let $w \in W$ and $\omega_j$ be a fundamental weight of $G^{sc}$. Let $V_{\omega_j}$ be the irreducible representation of $G^{sc}$ with highest weight $\omega_j$ and $v_{\omega_j}^+ $ be a fixed highest weight vector. Define for any $u \in U_-$:
	\[
		\Delta_{\omega_j,w \cdot \omega_j}(u) = \langle u^{sc}  \cdot v_{\omega_j}^+ , \bar w \cdot v_{\omega_j}^+ \rangle.
	\]
	Here $\langle u^{sc}  \cdot v_{\omega_j}^+ , \bar w \cdot v_{\omega_j}^+ \rangle=\langle \bar w^{-1}u^{sc}  \cdot v_{\omega_j}^+ ,  \cdot v_{\omega_j}^+ \rangle$ denotes the highest weight vector coefficient of $\bar w^{-1}u^{sc} \cdot v_{\omega_j}^+$ in terms of the weight space decomposition.
\end{defn}

Note that the smallest representative $w^P$ in $W$ of $[w_0] \in W/W_P$ has the following reduced expression:
\begin{equation}\label{eq:redexp}
	w^P = s_1 \dots s_{m-2} s_{m-1} s_m s_{m-2} \dots s_1.
\end{equation}
Here we state the result from \cite{GLS} applied to our particular setting.

\begin{theorem}[{\cite[Section 8]{GLS}}]\label{t:GLS}
	Consider the reduced expression $s_{i_1} \dots s_{i_{2m-2}} = s_1 \dots s_{m-2} s_m s_{m-1} s_{m-2} \dots s_1$ for $(\dot w^P)^{-1}$ coming from \eqref{eq:redexp}. The coordinate ring of the unipotent cell $U_-^P:= U_- \cap B_+ (\dot w^P)^{-1} B_+$ inside $\PSO_{2m}$ is 
	\[
		\C\left[U_-^P\right] = \C\left[\Delta_{\omega_{i_r},(\dot w^P)^{-1}_{\leq r} \cdot \omega_{i_r}},\Delta_{\omega_{2m-2-s},(\dot w^P)^{-1}_{\leq s} \cdot \omega_{2m-2-s}}^{-1}\right]
	\]
	where 
	\begin{itemize}
		\item $1 \leq r \leq 2m-2$; $m-1 \leq s \leq 2m-2$ ;
    	\item $(\dot w^P)^{-1}_{\leq r} := s_{i_1} \dots s_{i_r}$.
    \end{itemize}
\end{theorem}

If $j<m$ then $\Delta_{\omega_j,w \cdot \omega_j}(u)$ is a minor in the usual sense for the unique matrix $u^{\SO_{2m}}$ in $U_-^{\SO_{2m}}$ representing $u$. We denote the minor of $u^{\SO_{2m}}$ with row set $\{i_1,\dotsc, i_p\}$ and column set $\{j_1,\dotsc, j_p\}$ by $D_{j_1,\dots,j_p}^{i_1,\dots,i_p}(u)$. We now reformulate Theorem~\ref{t:GLS}  as follows.

\begin{cor}\label{cor:pres}
	The coordinate ring $\C\left[U_-^P\right]$ is generated by the minors
	\[
		D_{1,2,\dots,r}^{2,\dots,r,r+1}, \quad 1 \leq r \leq m-2; 
	\]
	\[
		D_{1,2,\dots,2m-1-s}^{2,\dots,2m-1-s,m+1}, \quad m+1 \leq s \leq 2m-3, \text{ and } D_1^{2m};
	\]
	the functions 
	\[
		\Delta_{\omega_m,\frac{1}{2}[-\epsilon_1+\epsilon_2+\dots+\epsilon_{m-1}-\epsilon_m]} \text{ and } \Delta_{\omega_{m-1},\frac{1}{2}[-\epsilon_1+\epsilon_2+\dots+\epsilon_m]},
	\]
which are Pfaffians; the inverses of minors
	\[
		\left(D_{1,2,\dots,2m-1-s}^{2,\dots,2m-1-s,m+1}\right)^{-1}, \quad m+1 \leq s \leq 2m-3, \text{ and } \left(D_1^{2m}\right)^{-1};
	\]
	and the inverses of Pfaffians 
	\[
		\Delta_{\omega_m,\frac{1}{2}[-\epsilon_1+\epsilon_2+\dots+\epsilon_{m-1}-\epsilon_m]}^{-1} \text{ and } \Delta_{\omega_{m-1},\frac{1}{2}[-\epsilon_1+\epsilon_2+\dots+\epsilon_m]}^{-1}.
	\]
\end{cor}

To relate the minors and Pfaffians of Corollary \ref{cor:pres} to the Pl\"ucker coordinates we will need to use a specific factorisation of generic elements of $U_-^P$. By an application of Bruhat's lemma \cite{Lusztig94}, a generic element in $U^P_-$ can be assumed to have a particular factorisation:
\begin{equation}\label{e:u2barfactb}
	\bar u_2 = y_1(a_1) \dots y_{m-2}(a_{m-2}) y_m(d) y_{m-1}(c) y_{m-2}(b_{m-2}) \dots y_1(b_1),
\end{equation}
where $a_i, c, d, b_j \ne 0$.

We have the following standard expression for the $\p_k$ on factorized elements, which is a simple consequence of their definition. 
\begin{lemma}\label{l:bislemma} 
	Fix $0 \leq k \leq 2m-2$ an integer. Then if $\bar u_2$ is of the form \eqref{e:u2barfactb} we have
	\begin{align*}
		\p_k(\bar u_2) = 	\begin{cases}
		    					1 & \text{if $k=0$,} \\
                    			a_1 \dots a_{k-1} (a_k + b_k) & \text{if $1 \leq k \leq m-2$,} \\
                    			a_1 \dots a_{m-2} c & \text{if $k=m-1$,} \\
                    			a_1 \dots a_{m-2} c d & \text{if $k=m$,} \\
                    			a_1 \dots a_{m-2} c d b_{m-2} \dots b_{2m-1-k} & \text{otherwise.}
                   			\end{cases}
 	\end{align*}
 	and
 	\[
  		\p_{m-1}'(\bar u_2) = a_1 \dots a_{m-2} d.\qed
 	\]
\end{lemma}

We can now prove the lemma we need. 

\begin{lemma}\label{l:iso}
	We have the following equalities of generalised minors and Pl\"ucker coordinates evaluated on $\bar u_2\in U^P_-$:
	\begin{align}
		& D_1^{2m}(\bar u_2) =\p_{2m-2}(\bar u_2),\label{e:minor1} \\ 
		& \Delta_{\omega_{m-1},\frac{1}{2}[-\epsilon_1+\epsilon_2+\dots+\epsilon_m]}(\bar u_2)=\p_{m-1} (\bar u_2), \label{e:minor2} \\
		& \Delta_{\omega_m,\frac{1}{2}[-\epsilon_1+\epsilon_2+\dots+\epsilon_{m-1}-\epsilon_m]} (\bar u_2) = \p_{m-1}'(\bar u_2),\label{e:minor3} \\
		& D_{1,2,\dots,2m-1-s}^{2,\dots,2m-1-s,m+1}(\bar u_2)=	\delta_{s-m}(\bar u_2), \text{ for $m+1\leq s \leq 2m-3$,}\label{e:minor4}
	\end{align}
 	where we recall that $\delta_{s-m} = \sum_{k=s}^m (-1)^{s-k} \p_{k-m} \p_{3m-2-k}$.
\end{lemma}

\begin{proof}
The identity \eqref{e:minor1} follows immediately from the definition of the Pl\"ucker coordinates. For the identity \eqref{e:minor2}, write
\[
	\Delta_{\omega_{m-1},\frac{1}{2}[-\epsilon_1+\epsilon_2+\dots+\epsilon_m]}(\bar u_2)=(D_{1,\dots,m-1,m+1}^{2,\dots,m,2m}(\bar u_2))^{\frac{1}{2}}.
\]
Note that in the definition of $\Delta_{\omega_{j},w\cdot \omega_j}$ we have chosen the representative $\bar w$ in such a way that evaluated on a factorized $\bar u_2$ the generalized minors will be nonnegative for any positive choice of the coordinates $a_i, b_i, c,d$ (i.e. on 'totally positive' $\bar u_2$). This determines the choice of square root.  
Then developing $D_{1,\dots,m-1,m+1}^{2,\dots,m,2m}(\bar u_2)$ with respect to the last column, we get
\[
	D_{1,\dots,m-1,m+1}^{2,\dots,m,2m}(\bar u_2) = D_{1,\dots,m-1}^{2,\dots,m}(\bar u_2) D^{2m}_{m+1}(\bar u_2) = p_{m-1}(\bar u_2) D_{1,\dots,m-1}^{2,\dots,m}(\bar u_2)
\]
using the definition of $p_{m-1}(\bar u_2)$. Finally, since the matrix is $\bar u_2$ orthogonal:
\[
	D_{1,\dots,m-1}^{2,\dots,m}(\bar u_2) = D_{1,\dots,m+1}^{1,\dots,m,2m}(\bar u_2).
\]
Developing again with respect to the last column, we obtain
\[
	D_{1,\dots,m+1}^{1,\dots,m,2m}(\bar u_2) = D_{1,\dots,m}^{1,\dots,m}(\bar u_2) D^{2m}_{m+1}(\bar u_2) = p_{m-1}(\bar u_2),
\]
using the definition of $p_{m-1}(\bar u_2)$ and the fact that $\bar u_2$ is lower unipotent. The identity~\eqref{e:minor2} then follows. The proof of the identity~\eqref{e:minor3} is similar.

Let us now prove the identity ~\eqref{e:minor4}. Developing $D_{1,2,\dots,2m-1-s}^{2,\dots,2m-1-s,m+1}(\bar u_2)$ with respect to the $(2m-1-s)$-th column, we see that it is equal to
\[
	D_{2m-1-s}^{m+1}(\bar u_2) D_{1,2,\dots,2m-2-s}^{2,\dots,2m-1-s}(\bar u_2)-D_{1,\dots,2m-2-s}^{2,\dots,2m-2-s,m+1}(\bar u_2).
\]
Since $\bar u_2$ is orthogonal for $Q$, we have
\[
	D_{1,2,\dots,2m-2-s}^{2,\dots,2m-1-s}(\bar u_2) = D_{1,\dots,s+2}^{1,\dots,s+1,2m}(\bar u_2),
\]
and since $\bar u_2$ is in $U_-$,
\[
	D_{1,\dots,s+2}^{1,\dots,s+1,2m}(\bar u_2) = D_{s+2}^{2m} (\bar u_2)= \p_{2m-2-s} (\bar u_2).
\]
Finally
\[
	D_{1,2,\dots,2m-1-s}^{2,\dots,2m-1-s,m+1}(\bar u_2) = D_{2m-1-s}^{m+1}(\bar u_2) \p_{2m-2-s}(\bar u_2) -D_{1,\dots,2m-2-s}^{2,\dots,2m-2-s,m+1}(\bar u_2),
\]
hence
\[
	D_{1,2,\dots,2m-1-s}^{2,\dots,2m-1-s,m+1}(\bar u_2) = \sum_{k=s}^{2m-2} (-1)^{s-k} D_{2m-1-s}^{m+1} (\bar u_2)\p_{2m-2-s} (\bar u_2).
\]
We also have $D_{2m-1-s}^{m+1} (\bar u_2)= d b_{2m-2} \dots b_{2m-1-s}$ for $m+1 \leq s \leq 2m-2$. Indeed, by definition
\[
	D_{2m-1-s}^{m+1} (\bar u_2) = \langle v_{m+1}^* \cdot \bar u_2, v_{2m-1-s} \rangle = d b_{2m-2} \dots b_{2m-1-s}.
\]
Hence
\begin{align*}
	D_{1,2,\dots,2m-1-s}^{2,\dots,2m-1-s,m+1}(\bar u_2) &= \sum_{k=s}^{2m-2} (-1)^{s-k} d b_{2m-2} \dots b_{2m-1-s} \p_{2m-2-s} \\
	&= \sum_{k=s}^m (-1)^{s-k} \p_{k-m}(\bar u_2) \p_{3m-2-k}(\bar u_2). \qedhere
\end{align*}
\end{proof}

\begin{proof}[Proof of Proposition~\ref{p:iso-even}]
Recall that $\pi_L=\phi\circ \theta$ where $\theta$ is the isomorphism constructed in Lemma~\ref{l:theta} and $\phi:U^P_-\to\X$ is the natural map $\bar u_2\mapsto P\bar u_2$. It remains to prove that $\phi$ is an isomorphism onto $\XX_{2m-2}$. We start by proving that the image of $\phi$ is contained in $\XX_{2m-2}$.

Indeed, if $\bar u_2 \in U_-^P$, then by  Corollary \ref{cor:pres}  the minors $D_{1,2,\dots,2m-1-s}^{2,\dots,2m-1-s,m+1}(\bar u_2)$ and $D_1^{2m}(\bar u_2)$ and the Pfaffians $\Delta_{\omega_m,\frac{1}{2}[-\epsilon_1+\epsilon_2+\dots+\epsilon_{m-1}-\epsilon_m)]}(\bar u_2)$ and $\Delta_{\omega_{m-1},\frac{1}{2}[-\epsilon_1+\epsilon_2+\dots+\epsilon_m)]}(\bar u_2)$ do not vanish. Since we have proved in Lemma \ref{l:iso} that those correspond precisely to the divisors involved in defining $\XX_{2m-2}$, it follows that $P\bar{u_2} \in \XX_{2m-2}$. We may now prove that $\phi$ is an isomorphism between $U^P_-$ and $\XX_{2m-2}$.

Injectivity of the pullback map $\phi^*:\C[\XX_{2m-2}] \to \C[U_-^P]$  is a simple consequence of the fact that the map $U_-^P \to \X_{2m-2}$ is dominant. We now prove that $\phi^*$ is surjective by observing that each of the functions  generating $\C[U_-^P]$  (as in Corollary~\ref{cor:pres}) has a preimage. 
 
We have already seen that the inverses of minors and Pfaffians correspond to the inverses of denominators of $\W$. Let us now consider the minors $D_{1,2,\dots,r}^{2,\dots,r,r+1}$ for $1 \leq r \leq m-2$ and $D_{1,2,\dots,2m-1-s}^{2,\dots,2m-1-s,m+1}$ for $m+1 \leq s \leq 2m-3$. In Lemma \ref{l:iso}, we proved that
\[
 	D_{1,2,\dots,2m-1-s}^{2,\dots,2m-1-s,m+1}=\phi^*(\delta_{s-m})
\]
and 
\begin{align*}
  	D_{1,\dots,r}^{2,\dots,r+1} = D_{1,\dots,2m-r}^{1,\dots,2m-1-r} = D_{2m-r}^{2m}=\phi^*(\p_r).
\end{align*}
Finally, $D_1^{2m}=\phi^*(\p_{2m-2})$, and the Pfaffians 
\[
 	\Delta_{\omega_m,\frac{1}{2}[-\epsilon_1+\epsilon_2+\dots+\epsilon_{m-1}-\epsilon_m)]} \text{ and }  \Delta_{\omega_{m-1},\frac{1}{2}[-\epsilon_1+\epsilon_2+\dots+\epsilon_m)]}.
\]
are pullbacks of the Pl\"ucker coordinates $p_{m-1}'$ and $p_{m-1}$, by Lemma~\ref{l:iso}. This concludes the proof.
\end{proof}

\subsection{Comparison of the superpotentials\label{s:Isorat}}

In this section we will prove Theorem~\ref{t:even}. We saw in the previous section that $\pi=\pi_L\circ \pi_R\inv:\RR \to \XX_{2m-2}$ is an isomorphism. Note that we have a commutative diagram 

\begin{center}
	\begin{tikzpicture}
  		\matrix [matrix of math nodes,row sep=1cm, column sep=1cm,text height=1.5ex,text depth=0.25ex]
  			{
	  			|(Z)| Z_d & & |(R)| \RR \\
	 			& |(C)| \C & \\
			};
		\path[->]
				(Z) edge node[above] {$\pi_R$} (R)
		;
		\path[->]
			(Z) edge node[below] {$\sim$} (R)
		;
		\path[->]
			(Z) edge node[below]  {$F_q$}  (C) 
		;
		\path[->]
			(R) edge node[below,xshift=0.5 cm] {$\F$}  (C)
		;
	\end{tikzpicture}
\end{center}
Therefore 
\[
	(\pi\inv)^*(\F)=(\pi_L\inv)^*(F_q).
\]
This gives a regular function on $\XX_{2m-2}$ which we denote by  $\widetilde \W$. The statement of Theorem~\ref{t:even} says that $\widetilde \W$ and $\W$ agree. We will prove this by expressing both functions in terms of coordinates introduced earlier. Namely we consider the set of factorized elements $P\bar u_2$ with $\bar u_2$ as in  \eqref{e:u2barfactb} with nonzero coordinates $a_i,b_i, c,d$ as defining an open dense subvariety inside $\XX_{2m-2}$ which is isomorphic to a torus. We call this subvariety $\XX_{2m-2}$. To finish the proof we will show that the restrictions of $\widetilde \W$ and of $\W$ to $\XX_{2m-2}$ agree. This will additionally give an interesting Laurent polynomial formula for the superpotential, which we will use in Section~\ref{s:aseries} to describe a flat section of the Dubrovin connection. 

\begin{prop}\label{p:W1} 
	$\widetilde \W$ and $\W$ restricted to a particular torus $\XXlus$ inside $\XX_{2m-2}$ have the following Laurent polynomial expression 
  	\begin{equation*}
     	\HH = a_1 + \dots + a_{m-2} + c + d + b_{m-2} + \dots + b_1 + q \frac{a_1+b_1}{a_1 \dots a_{m-2} c d b_{m-2} \dots b_1}.
  	\end{equation*}
\end{prop}
We call $(\XXlus,\HH)$ the \emph{quiver mirror}. To prove Proposition \ref{p:W1} we will need the following:
\begin{lemma}\label{p:eis}
	If $u_1\in U_+$, $\bar u_2 \in U_-$, $u_1 d \dot w_P \bar u_2\in Z_d$, and $\bar u_2$ can be written as in \eqref{e:u2barfactb}, then we have the following identities:
	\begin{align}\label{eq:ei1first}
  		f_i^*(\bar u_2) = 	\begin{cases}
                     			a_i + b_i & \text{if $1 \leq i \leq m-2$,} \\
                     			c & \text{if $i=m-1$,} \\
                     			d & \text{if $i=m$.} \\
                    		\end{cases}
 	\end{align}
 
 	\begin{align}\label{eq:ei2first}
   		e_i^* (u_1) = 	\begin{cases}
                  			0 & \text{if $2 \leq i \leq m$,} \\
                 			 q \frac{a_1 + b_1}{a_1 \dots a_{m-1} c d b_{m-1} \dots b_1} & \text{if $i=1$}.
                 		\end{cases}
	 \end{align}
\end{lemma}
 
\begin{proof}
	Equation \eqref{eq:ei1first} is obtained immediately from the definition of $\bar u_2$. For Equation \eqref{eq:ei2first}, notice that
	\begin{align*}
  			e_i^*(u_1) 	& = \frac{\langle u_1^{-1} \cdot v_{\omega_i}^- , e_i \cdot v_{\omega_i}^-\rangle}{\langle u_1^{-1} \cdot v_{\omega_i}^- , v_{\omega_i}^-\rangle} \\
	     				& = \frac{\langle d \dot w_P \bar u_2 \cdot v_{\omega_i}^+ , e_i \cdot v_{\omega_i}^-\rangle}{\langle d \dot w_P \bar u_2 \cdot v_{\omega_i}^+ , v_{\omega_i}^-\rangle}.
 \end{align*}
Assume $2 \leq i \leq m$. Then $e_i^*(u_1) = 0$ if and only if $\langle \bar u_2 \cdot v_{\omega_i}^+ , \dot w_P^{-1} e_i \cdot v_{\omega_i}^-\rangle = 0$. Now the vector $w_P^{-1} e_i \cdot v_{\omega_i}^-$ is in the $\mu$-weight space of the $i$-th fundamental representation, where $\mu = w_P^{-1} s_i(-\omega_i)$. Moreover, $\bar u_2 \in B_+ (\dot w^P)^{-1} B_+$, hence $\bar u_2 \cdot v_{\omega_i}^+$ can have non-zero components only down to the weight space of weight $(w^P)^{-1}(\omega_i) = w_P^{-1}(-\omega_i)$. Since $l(w_P^{-1} s_i) > l(w_P^{-1})$ for $2 \leq i \leq m$, this is higher than $\mu$, which proves that $e_i^*(u_1) = 0$.

Now assume $i=1$. We have
\begin{align*}
  	e_1^*(u_1) 	& = \frac{\langle d \dot w_P \bar u_2 \cdot v_{\omega_1}^+ , e_1 \cdot v_{\omega_1}^-\rangle}{\langle d \dot w_P \bar u_2 \cdot v_{\omega_1}^+ , v_{\omega_1}^-\rangle} \\
	     		& = (\omega_1 + \alpha_1 - \omega_1)(d) \frac{\langle \bar u_2 \cdot v_{\omega_1}^+ , \dot w_P^{-1} e_1 \cdot v_{\omega_1}^-\rangle}{\langle \bar u_2 \cdot v_{\omega_1}^+ , \dot w_P v_{\omega_1}^-\rangle} \\
             	& = q \frac{\langle \bar u_2 \cdot v_{\omega_1}^+ , \dot w_P^{-1} e_1 \cdot v_{\omega_1}^-\rangle}{\langle \bar u_2 \cdot v_{\omega_1}^+ , v_{\omega_1}^-\rangle}.
\end{align*}
First look at the denominator. The only way to go from the highest weight vector $v_{\omega_1}^+$ of the first fundamental representation to the lowest weight vector $v_{\omega_1}^-$ is to apply $g \in B_+ w B_+$ for $w \geq (w^P)^{-1}$. Since $\bar u_2 \in B_+ (\dot w^P)^{-1} B_+$, it follows that we need to take all factors of $\bar u_2$, and normalising $ v_{\omega_1}^-$ appropriately, we get
\[
 	 \langle \bar u_2 \cdot v_{\omega_1}^+ , v_{\omega_1}^-\rangle = a_1 \dots a_{m-1} c d b_{m-1} \dots b_1.
\]
Finally, we look at the numerator $\langle \bar u_2 \cdot v_{\omega_1}^+ , \dot w_P^{-1} e_1 \cdot v_{\omega_1}^-\rangle$. The vector $\dot w_P^{-1} e_1 \cdot v_{\omega_1}^-$ has weight 
\[
  	\mu' = \dot w_P^{-1} s_1 (-\omega_1) = \dot w_P^{-1} (-\epsilon_2) = \epsilon_2.
\]
Write $w_P^{-1} s_1$ as a prefix $w' = s_1 s_2 \dots s_{m-2} s_m s_{m-1} s_{m-2} \dots s_2$ of $(w^P)^{-1}$. We have $w' s_1 = (w^P)^{-1}$, hence the way from $v_{\omega_1}^+$ to $w' \cdot v_{\omega_1}^-$ is through $s_1$. From the factorization of $\bar u_2$ in \eqref{e:u2barfactb}, it follows that $\langle \bar u_2 \cdot v_{\omega_1}^+ , \dot w_P^{-1} e_1 \cdot v_{\omega_1}^-\rangle = a_1 + b_1$.
\end{proof}

\begin{proof}[Proof of Proposition~\ref{p:W1}]
Using the expression \eqref{e:F} of the superpotential from \cite{rietsch}, we immediately deduce expression for $\widetilde \W$ as a Laurent polynomial from Lemma~\ref{p:eis}.
\end{proof}

Next, using  Lemma~\ref{l:bislemma} and Proposition~\ref{p:W1}, we express $\widetilde \W$ in terms of Pl\"ucker coordinates and deduce the theorem.

\begin{proof}[Proof of Theorem~\ref{t:even}]
From Lemma~\ref{l:bislemma}, it follows that for $\bar u_2$ as in \eqref{e:u2barfactb}
\[
 	 \p_{\ell+1} (\bar u_2) \p_{2m-2-\ell} (\bar u_2)= (a_{\ell+1} + b_{\ell+1})(a_1 \dots a_\ell)^2 a_{\ell+1} \dots a_{m-2} c d b_{m-2} \dots b_{\ell+1}
\]
for $0 \leq \ell \leq m-3$.  We also get that $\p_k(\bar u_2) \p_{2m-2-k} (\bar u_2)$ is equal to
\begin{equation}\label{eq:sum} 
	  \begin{cases}
			a_1 \dots a_{m-2} cd b_{m-2} \dots b_1 &\mbox{if }k=0;\\
     		(a_1+b_1) a_1 \dots a_{m-2} cd b_{m-2} \dots b_2 &\mbox{if }k=1;\\
    		(a_k + b_k) (a_1 \dots a_{k-1})^2 a_k \dots a_{m-2} c d b_{m-2} \dots b_{k+1}
      		&\mbox{if }2 \leq k \leq m-3. \end{cases}
\end{equation} 
Using \eqref{eq:sum}, we find that most terms in $\delta_\ell(\bar u_2)=\sum_{k=0}^\ell (-1)^k \p_{\ell-k} (\bar u_2) \p_{2m-2+k-\ell}(\bar u_2)$ cancel, and
\[
 	 \delta_\ell(\bar u_2) = (a_1 \dots a_\ell)^2 a_{\ell+1} \dots a_{m-2} c d b_{m-2} \dots b_{\ell+1}.
\]
This proves that 
\[
  	\frac{\p_{\ell+1} \p_{2m-2-\ell}}{\delta_\ell} (\bar u_2) = a_{\ell+1} + b_{\ell+1} 
\]
for $0 \leq \ell \leq m-3$. Moreover:
\[
  	\frac{\p_m}{\p_{m-1}}(\bar u_2) = \frac{a_1 \dots a_{m-2} c d}{a_1 \dots a_{m-2} c} = d,
\]
and
\[
  	\frac{\p_m}{\p_{m-1}'}(\bar u_2) = \frac{a_1 \dots a_{m-2} c d}{a_1 \dots a_{m-2} d} = c.
\]
For the first and last terms, we obtain
\[
  	\frac{\p_1}{\p_0} (\bar u_2) = a_1 + b_1,
\]
and
\[
  	\frac{\p_1}{\p_{2m-2}} (\bar u_2) = \frac{a_1+b_1}{a_1 \dots a_{m-1} c d b_{m-1} \dots b_1}
 \]
as easy consequences of Lemma~\ref{l:bislemma}. Using Proposition~\ref{p:W1}, this proves that $\widetilde \W$ coincides with the definition of $\W$ from Equation~\eqref{eq:Weven}
\begin{equation*}
      \W = \frac{\p_{1}}{\p_0} + \sum_{\ell=1}^{m-3} \frac{\p_{\ell+1} \p_{2m-2-\ell}}{\delta_\ell} + \frac{\p_m}{\p_{m-1}} + \frac{\p_m}{\p_{m-1}'} + q\frac{\p_{1}}{\p_{2m-2}}.\qedhere
\end{equation*} 
\end{proof}

\subsection{Comparison with the Givental and Laurent polynomial mirrors for even quadrics}\label{comparison:even}

Let us recall the Laurent polynomial LG model of $\Q_{2m-2}$ from Equation \eqref{e:toricLG}
\[
	\LL = z_1 + \dots + z_{2m-3} + \frac{(z_{2m-2}+q)^2}{z_1 z_2 \dots z_{2m-2}},
\]
defined over the torus
\[
	\TT:=\left\{ (z_1,\dots,z_{2m-2}) \mid z_i \neq 0\quad \forall\; i\right\},
\] 
and the Givental LG model from Equation \eqref{e:giventalLG}
\[
	\G = \nu_1 + \dots + \nu_{2m-2},
\]
defined over the affine variety 
\[
	\VV=\left\{ (\nu_1,\dots,\nu_{2m}) \mid \nu_i \neq 0\; \forall\; i,\,\  \prod_{i=1}^{2m}\nu_i=q,\ \ \nu_{2m-1}+\nu_{2m}=1\right\}.
\]
These two LG models are related by a birational change of coordinates analogous to that of \cite[Rmk.~19]{przyjalkowski}, namely
\begin{align*}
	z_i = \begin{cases}
			\nu_{i+1} & \text{for $1 \leq i \leq 2m-3$;} \\
			q\frac{\nu_{2m-1}}{\nu_{2m}} & \text{for $i=2m-2$;} 
		  \end{cases}
\end{align*}
and conversely
\begin{align*}
	\nu_i = \begin{cases}
			\frac{(z_{2m-2}+q)^2}{z_1 \dots z_{2m-2}} & \text{for $i=1$;} \\	
			z_{i-1} & \text{for $2 \leq i \leq 2m-2$;} \\
			\frac{z_{2m-2}}{z_{2m-2}+q} & \text{for $i=2m-1$;} \\
			\frac{q}{z_{2m-2}+q} & \text{for $i=2m$.}
		  \end{cases}
\end{align*}
This change of variables defines an isomorphism
\[
	\TT \setminus \{ z_{2m-2}+q=0 \} \cong \VV
\]
which identifies the superpotentials $\LL$ and $\G$.

Let us now compare these two LG models with ours. Consider the change of coordinates
\begin{align*}
	z_i = \begin{cases}
			\frac{\p_i}{\p_{i-1}} & \text{for $1 \leq i \leq m-2$;} \\
			\frac{\p_{2m-3-i}\delta_{2m-5-i}}{\p_{2m-4-i}\delta_{2m-4-i}} & \text{for $m-1 \leq i \leq 2m-5$;} \\
			\frac{\p_m}{\p_{m-1}} & \text{for $i=2m-4$;} \\
			\frac{\p_m}{\p_{m-1}'} & \text{for $i=2m-3$;} \\
			q\frac{\delta_{m-3}}{\delta_{m-2}} & \text{for $i=2m-2$.} 
		  \end{cases}
\end{align*}
It is well-defined on the following intersection $\tilde T$ of two cluster tori
\[
	\tilde T:=\{x \in \XXcan \ \vert \  \p_i(x) \neq 0 \text{ for all } 0 \leq i \leq m-2 \text{ and } \p_m(x) \neq 0 \}.
\] 
The inverse change of coordinates is given by
\begin{align*}
	\p_i = \begin{cases}
			z_1 \dots z_i & \text{for $1 \leq i \leq m-2$;} \\
			q z_1 \dots z_{m-2} \frac{z_{2m-3}}{z_{2m-2}} & \text{for $i=m-1$;} \\
			q z_1 \dots z_{m-2} \frac{z_{2m-4}z_{2m-3}}{z_{2m-2}} & \text{for $i=m$;} \\
			q z_1 \dots z_{i-2} \left( 1+ \frac{z_{2m-1-i}}{z_{i-2}}\right) \frac{z_{2m-4}z_{2m-3}}{z_{2m-2}+q} & \text{for $m+1 \leq i \leq 2m-3$;} \\
			q z_1 \dots z_{2m-3} \frac{z_1}{z_{2m-2}+q} & \text{for $i=2m-2$.} 
		  \end{cases}
\end{align*}
and $\p_{m-1}' = q z_1 \dots z_{m-2} \frac{z_{2m-4}}{z_{2m-2}}$. Moreover, we have
\begin{align*}
	\delta_j = \begin{cases}
			\frac{z_2 \dots z_{j+1}}{z_{2m-4-j} \dots z_{2m-5}} & \text{for $1 \leq j \leq m-3$;} \\
			\frac{q}{z_{2m-2}} \cdot \frac{z_2 \dots z_{m-2}}{z_{m-1} \dots z_{2m-5}} & \text{for $i=2m-2$.} 
		  \end{cases}
\end{align*}
We see that the inverse change of coordinates is well-defined over $\TT \setminus \{ z_{2m-2}+q=0 \}$, which is isomorphic to the Givental mirror manifold $\VV$. Hence we obtain an isomorphism
\[
	\XXcan \supset \tilde T \cong \TT \setminus \{ z_{2m-2}+q=0 \} \cong \VV \subset \TT
\]
which identifies the (restrictions of) the superpotentials $\W$ and $\LL$. It also identifies the form $\omegaGiv$ with $\omegaCan$. This proves Proposition \ref{prop:Giv2W} from the introduction in the case of even quadrics.

\subsection{The critical points of the canonical mirror}

Since the canonical mirror $(\XXcan,\W)$ is isomorphic to the Lie-theoretic mirror $(\RR,\F)$, it follows from \cite{rietsch} that $\W$ has the `correct' number of critical points on $\XXcan$, that is, $\dim H^*(\Q_{2m-2},\C)=2m$. Here we give explicit expression for the critical points, and compare with the critical points of the classical mirrors $(\VV,\G)$ and $(\TT,\LL)$.

\begin{prop}
\label{p:critical:even}
	The critical points of the superpotential $\W$ on $\XXcan$ are given by
	\begin{align*}
		\p_j = 	\begin{cases}
					\zeta^j & \text{if $1 \leq j \leq m-2$ ;} \\
					\frac{1}{2}\zeta^j & \text{if $m-1 \leq j \leq 2m-3$ ;} \\
					q & \text{if $j=2m-2$},
				\end{cases}
	\end{align*}
	and $\p_{m-1}'=\frac{1}{2}\zeta^{m-1}$, where $\zeta$ is a primitive $(2m-2)$-st root of $4q$. The associated critical value is $(2m-2)\zeta$.
	Moreover there are two extra critical points given by $\p_1 = \dots =\p_{m-2} = \p_m = \p_{2m-3} = 0$, $\p_{m-1}=-\p_{m-1}'=\pm\sqrt{q}$, $\p_{2m-2}=-q$ with corresponding critical value $0$. These two critical points do not belong to $\TT$, $\VV$ or $\XXlus$.
\end{prop}

\begin{proof}
The proof is very similar to that of Proposition \ref{p:critical-odd} and we don't repeat it here.
\end{proof}

\section{The quiver mirrors $(\XXlus,\HH)$}\label{sec:quiver}

In this section we will explain how our quiver superpotential $\HH$ for $Q_N$ can be read off from a certain quiver, justifying its name. This is analogous to the type $A$ complete flag variety case
\cite{Givental:QToda} and partial flag variety case \cite{BCFKvSGr,BCFKvSPF}, where one can also read off Laurent polynomial superpotentials from quivers.

We begin by explaining the \cite{BCFKvSGr} formula for the Grassmannian $Gr_2(4)$. Note that since $Gr_2(4)$ is defined by a single (quadratic) Pl\"ucker relation, it is isomorphic to the quadric $Q_4$. 

For $Gr_2(4)$ the quiver from \cite{BCFKvSGr} is shown in Figure~\ref{fig:Gr24}. The Laurent polynomial superpotential can be read off easily. There are two versions. In the left hand picture the coordinates $t_{ij}$ of the torus $(\C^*)^4$ are in bijection with vertices of the quiver. To each  arrow we associate a Laurent monomial by taking the coordinate at the head of the arrow divided by the coordinate at the tail. The Laurent polynomial corresponding to the quiver is the sum of all of the Laurent monomials associated to the arrows. 

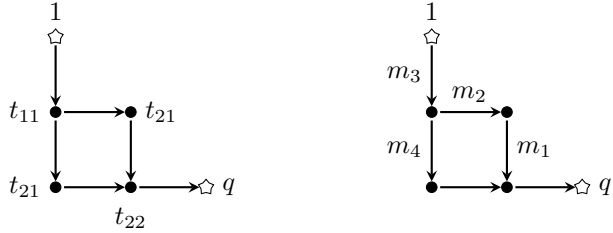
\begin{figure}[htbp]
\begin{tikzpicture}
	\tikzstyle{every node}=[draw,circle,fill=black,minimum size=4pt,
                            inner sep=0pt,label distance=2pt]
                            
	\node[star,fill=white,label=above:{1},minimum size=6pt] (1) at (0,0) {};
	\node[label=left:{$t_{11}$}] (11) at (0,-1) {};
	\node[label=left:{$t_{21}$}] (21) at (0,-2) {};
	\node[label=right:{$t_{21}$}] (12) at (1,-1) {};
	\node[label=below:{$t_{22}$}] (22) at (1,-2) {};
	\node[star,fill=white,label=right:{$q$},minimum size=6pt] (q) at (2,-2) {};

	\draw[->,thick,>=stealth,shorten <=1pt]
	(1) edge (11)
	(11) edge (12)
	(11) edge (21)
	(12) edge (22)
	(21) edge (22)
	(22) edge (q);
	
	\node[star,fill=white,label=above:{1},minimum size=6pt] (1p) at (5,0) {};
	\node (11p) at (5,-1) {};
	\node (21p) at (5,-2) {};
	\node (12p) at (6,-1) {};
	\node (22p) at (6,-2) {};
	\node[star,fill=white,label=right:{$q$},minimum size=6pt] (qp) at (7,-2) {};

	\path[->,thick,>=stealth,every node/.style={font=\normalsize},shorten <=1pt]
	(1p) edge node[left] {$m_3$} (11p)
	(11p) edge node[above] {$m_2$} (12p)
	(11p) edge node[left] {$m_4$} (21p)
	(12p) edge node[right] {$m_1$} (22p)
	(21p) edge (22p)
	(22p) edge (qp);
\end{tikzpicture}
	\caption{The quiver for $Gr_2(4)$ and two choices of coordinates.}
	\label{fig:Gr24}
\end{figure}

The labels $m_i$ of the arrows in the right hand version are another natural choice of coordinates on the same torus. Indeed these are coordinates related to factorizations into one-parameter subgroups of Lie-theoretic mirrors used in \cite{Lusztig94}, compare \cite{MR}. We suppose the remaining arrows are labelled in such a way that the square commutes and any path leading from $1$ to $q$ has labels whose product equals $q$. These are Laurent monomials in the variables $m_i$ (depending on $q$). Then the Laurent polynomial superpotential is obtained in \cite{BCFKvSGr} as the sum of the labels of all of the arrows of the quiver. In the case of $Gr_2(4)$ it is
\begin{equation}\label{e:Gr24}
	m_1+m_2+m_3+m_4+\frac{m_1 m_2}{m_4}+q\frac{1}{m_1 m_2 m_3}.
\end{equation}
Since $Gr_2(4)$ is isomorphic to $Q_4$, this suggests it should be related to the superpotential $(\XXlus, \HH)$ from \eqref{e:quiverMirrorEven} for $Q_4$,
\begin{equation}\label{e:Q4}
	a_1+c+d+b_1+q \frac{a_1+b_1}{a_1 b_1cd}.
\end{equation}
There is indeed a toric change of coordinates turning Equation~\eqref{e:Gr24} into Equation~\eqref{e:Q4}:
\[
	m_1 \mapsto \frac{q}{a_1 c d};\ m_2 \mapsto a_1;\ m_3 \mapsto c;\ m_4 \mapsto b_1.
\]
Note that the torus of the other Laurent polynomial mirror $(\TT,\LL)$ for $Q_4$ is a different one, as seen in Section~\ref{comparison:even}.

The superpotential \eqref{e:Q4} also comes from a quiver, see Figure~\ref{fig:quiverQ4}.  
\begin{figure}[htbp]
\begin{tikzpicture}
	\tikzstyle{every node}=[draw,circle,fill=black,minimum size=4pt,
                            inner sep=0pt,label distance=2pt]
                            
	\node[star,fill=white,label=above:{1},minimum size=6pt] (1) at (0,0) {};
	\node (2) at (0,-1) {};
	\node (3) at (0,-2) {};
	\node (4) at (0,-3) {};
	\node (5) at (1,-3) {};
	\node[star,fill=white,label=right:{$q$},minimum size=6pt] (q) at (1,-4) {};

	\path[->,thick,>=stealth,every node/.style={font=\normalsize},shorten <=1pt]
	(1) edge node[left] {$d$} (2)
	(2) edge node[left] {$c$} (3)
	(3) edge node[left] {$b_1$} (4)
	(3) edge node[right] {$a_1$} (5)
	(4) edge (q)
	(5) edge (q);
\end{tikzpicture}
	\caption{The quiver for $Q_4$.}
	\label{fig:quiverQ4}
\end{figure}
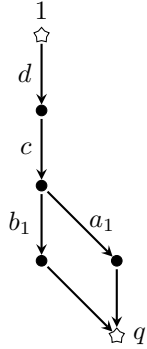
This generalises to all quadrics $Q_N$. Indeed our Laurent polynomial superpotentials \eqref{e:quiverMirrorOdd} and \eqref{e:quiverMirrorEven} for $Q_N$ can be described using quivers as in Figure \ref{fig:quiverQN}. The factorisation of $\bar u_2$ from \eqref{e:u2barfactb} can also be naturally read off the quiver (compare with \cite[Section~5.3]{MR}). This goes as follows.

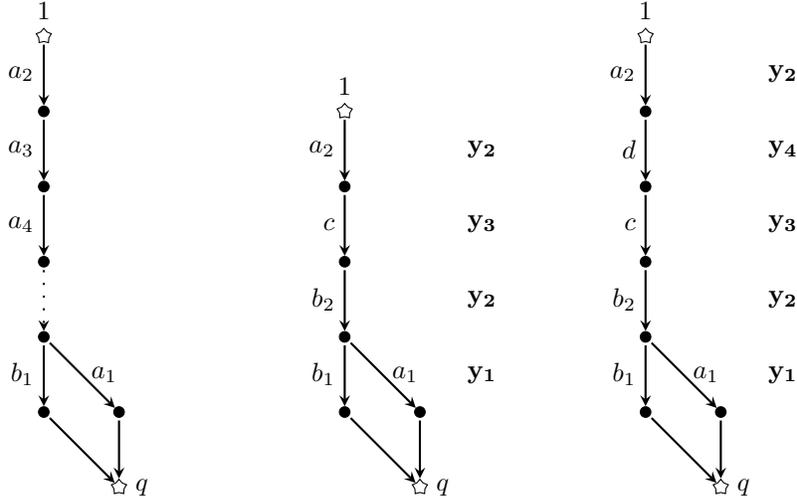
\begin{figure}[htbp]
\begin{tikzpicture}
	\tikzstyle{every node}=[draw,circle,fill=black,minimum size=4pt,
                            inner sep=0pt,label distance=2pt]
                            
	\node[star,fill=white,label=above:{1},minimum size=6pt] (1) at (0,0) {};
	\node (2) at (0,-1) {};
	\node (3) at (0,-2) {};
	\node (4) at (0,-3) {};
	\node (5) at (0,-4) {};
	\node (6) at (0,-5) {};
	\node (7) at (1,-5) {};
	\node[star,fill=white,label=right:{$q$},minimum size=6pt] (q) at (1,-6) {};

	\path[->,thick,>=stealth,every node/.style={font=\normalsize},shorten <=1pt]
	(1) edge node[left] {$a_2$} (2)
	(2) edge node[left] {$a_3$} (3)
	(3) edge node[left] {$a_4$} (4)
	(4) edge[loosely dotted] (5)
	(5) edge node[left] {$b_1$} (6)
	(5) edge node[right] {$a_1$} (7)
	(6) edge (q)
	(7) edge (q);
	
	\node[star,fill=white,label=above:{1},minimum size=6pt] (1a) at (4,-1) {};
	\node (2a) at (4,-2) {};
	\node (3a) at (4,-3) {};
	\node (4a) at (4,-4) {};
	\node (5a) at (4,-5) {};
	\node (6a) at (5,-5) {};
	\node[star,fill=white,label=right:{$q$},minimum size=6pt] (qa) at (5,-6) {};

	\path[->,thick,>=stealth,every node/.style={font=\normalsize},shorten <=1pt]
	(1a) edge node[left] {$a_2$} node[label={[label distance=1.5cm]right:$\mathbf{y_2}$}] {} (2a)
	(2a) edge node[left] {$c$} node[label={[label distance=1.5cm]right:$\mathbf{y_3}$}] {} (3a)
	(3a) edge node[left] {$b_2$} node[label={[label distance=1.5cm]right:$\mathbf{y_2}$}] {} (4a)
	(4a) edge node[left] {$b_1$} (5a)
	(4a) edge node[right] {$a_1$} node[label={[label distance=1cm]right:$\mathbf{y_1}$}] {} (6a)
	(5a) edge (qa)
	(6a) edge (qa);
	
	\node[star,fill=white,label=above:{1},minimum size=6pt] (1b) at (8,0) {};
	\node (2b) at (8,-1) {};
	\node (3b) at (8,-2) {};
	\node (4b) at (8,-3) {};
	\node (5b) at (8,-4) {};
	\node (6b) at (8,-5) {};
	\node (7b) at (9,-5) {};
	\node[star,fill=white,label=right:{$q$},minimum size=6pt] (qb) at (9,-6) {};

	\path[->,thick,>=stealth,every node/.style={font=\normalsize},shorten <=1pt]
	(1b) edge node[left] {$a_2$} node[label={[label distance=1.5cm]right:$\mathbf{y_2}$}] {} (2b)
	(2b) edge node[left] {$d$} node[label={[label distance=1.5cm]right:$\mathbf{y_4}$}] {} (3b)
	(3b) edge node[left] {$c$} node[label={[label distance=1.5cm]right:$\mathbf{y_3}$}] {} (4b)
	(4b) edge node[left] {$b_2$} node[label={[label distance=1.5cm]right:$\mathbf{y_2}$}] {} (5b)
	(5b) edge node[left] {$b_1$} (6b)
	(5b) edge node[right] {$a_1$} node[label={[label distance=1cm]right:$\mathbf{y_1}$}] {}(7b)
	(6b) edge (qb)
	(7b) edge (qb);
	
\end{tikzpicture}
	\caption{The quiver for $Q_N$, and the labelled quivers for $Q_5$ and $Q_6$.}
	\label{fig:quiverQN}
\end{figure}

Let the $N-2$ vertical arrows on the left-hand edge be labelled from top to bottom by $a_2, a_3, \dots, a_{m-1}, c, b_{m-1}, \dots, b_2$ for odd quadrics $Q_{2m-1}$, and by $a_2, a_3, \dots, a_{m-2},$ $d, c, b_{m-2}, \dots, b_2$ for even quadrics $Q_{2m-2}$. The diagonal arrow with the same tail as $b_1$ is labelled by $a_1$. The arrows below are not labelled. The labelled arrows can be organized into `levels' starting with $a_1, b_1$ at the bottom level. The levels are associated to the one-parameter subgroups $y_i$ (of $\PSO_{2m}$ for $X=Q_{2m-2}$, respectively of $\PSp_{2m}$ for $X=Q_{2m-1}$) as shown in the $Q_5$ and $Q_6$ examples. Reading off column by column from right to left and from top to bottom we recover the factorization \eqref{e:u2barfactb}.

\begin{remark}
	It is interesting to note that our quivers (restricted to the vertices which are not labelled  by
$q$) are orientations of type $D$ Dynkin diagrams with a special vertex added at either end. So we have three ways to associate a Dynkin diagram to a quadric: the type of its symmetry group, the type of the cluster algebra associated to the coordinate ring of its mirror, and the type of the quiver defining
its superpotential. See Table \ref{table:Dynkin}.
\end{remark}

\begin{center}
	\renewcommand{\arraystretch}{1.2}
	\begin{table}[h]
		\begin{tabular}{| l | l || l | l |}
		\hline
		Quadric & Symmetry group & Cluster type of mirror & Superpotential Quiver \\
		\hline
		\hline
		$Q_3$ & $B_2$ & $A_1$ & $D_3$ \\ \hline
		$Q_4$ & $D_3$ & $A_1$ & $D_4$ \\ \hline
		$Q_5$ & $B_3$ & $A_1^2$ & $D_5$ \\ \hline
		$Q_6$ & $D_4$ & $A_1^2$ & $D_6$ \\ \hline
		$Q_7$ & $B_4$ & $A_1^3$ & $D_7$ \\ \hline
		$\vdots$ & $\vdots$ & $\vdots$ & $\vdots$ \\ \hline
		\end{tabular}
		\vspace{.2cm}
		\caption{Dynkin diagrams associated to quadrics}
		\label{table:Dynkin}
	\end{table}
\end{center}

\section{The A-model and B-model connections}\label{sec:connections}

Our expression for the canonical LG model $\W$ in terms of homogeneous coordinates coming from $\XXcan\subset \P(H^*(X,\C)^*)$ makes it possible to compare in a very natural way the (small) Dubrovin connection on the $A$ side and the Gauss-Manin connection on the $B$ side. We recall first the relevant definitions on the $A$ side.

Let $X=\Q_N$. Consider  $H^*(X,\C[\hbar, q])$ as a space of sections on a trivial bundle with fiber $H^*(X,\C)$, over the base $\C_{\hbar} \times \C_{q}$, where the $\hbar$ and $q$ are the coordinates.
Let $\operatorname{Gr}$ be the operator on sections defined on the fibres as the `grading operator' $H^*(X,\C)\to H^*(X,\C)$ which multiplies $\sigma\in H^{2k}(X,\C)$ by $k$. We define the Dubrovin connection by
\begin{eqnarray}
	{}^A\nabla_{q\partial_q} S&:=& q\frac{\partial S}{\partial q} + \frac{1}{\hbar} \sigma_{1}\star_{q}S, \\
	{}^A \nabla_{\hbar\partial_\hbar} S&:=& \hbar\frac{\partial S}{\partial\hbar} - \frac{1}{\hbar} c_1(TX)\star_{q}S+ \operatorname{Gr} (S),
\end{eqnarray}
following the conventions of Iritani~\cite{iritani}, where $\star_q$ denotes the quantum cup product in the quantum cohomology, and $S$ may be any meromorphic or formal section of the above vector bundle. The above defines a meromorphic connection which is flat, see also \cite{Dub:2DTFT,Givental:EqGW,CoxKatz}. It therefore turns $H^*(X,\C[\hbar^{\pm 1}, q^{\pm 1}])$ into a $D$-module for $\C[\hbar^{\pm 1}, q^{\pm 1}]\langle\partial_\hbar,\partial_q\rangle$, which we will call  the \emph{$A$-model $D$-module} and denote by $M_A$.  Explicitly 
\begin{equation}\label{e:MA}
M_A:=H^*(X,\C[\hbar^{\pm 1}, q^{\pm 1}]),  \text{ with }  \partial_\hbar \sigma := {}^A\nabla_{\partial_\hbar}\sigma \text{ and } \partial_q \sigma := {}^A\nabla_{\partial_q}\sigma.
\end{equation}
This is the  $D$-module we consider on the $A$-model side.

We now define the $D$-module $M_B$. Let $\Omega^k(\XXcan)$ denote the space of all algebraic $k$-forms on $\XXcan$.

\begin{defn}\label{d:GM}
	Define the $\C[\hbar,q]$-module
	\[
		G_0^{\W}:=\Omega^n(\XXcan)[\hbar,q]/( \hbar d +  d\W\wedge - ) \Omega^{n-1}(\XXcan)[\hbar,q].
	\]
	It has a meromorphic (Gauss-Manin) connection given by
	\begin{eqnarray}
		{}^B\nabla_{q\partial_{q}} [\alpha]&=&q\frac \partial{\partial q} [\alpha] + \frac{1}{\hbar}\left[q\frac{\partial \W}{\partial q}\alpha\right],\\
		{}^B\nabla_{\hbar\partial_\hbar} [\alpha]&=&\hbar \frac\partial{\partial \hbar} [\alpha] -\frac{1}{\hbar} [ \W \alpha].
	\end{eqnarray}
	Let $M_B=G_0^{\W}\otimes_{\C[\hbar,q]}\C[\hbar^{\pm 1}, q^{\pm 1}]$. We view $M_B$ as a $\C[\hbar^{\pm 1},q^{\pm 1}]\langle\partial_\hbar,\partial_q\rangle$-module with $\partial_q$ acting by ${}^B\nabla_{\partial_{q}}$ and $\partial_{\hbar}$ acting by ${}^B\nabla_{\partial_\hbar}$. 
\end{defn}
On the $A$-model side a special role is played by the element $1\in M_A$ corresponding to the identity in $H^*(X,\C)$. For the $B$-model there is also a distinguished element. Recall that $\XXcan$ is  the complement of an anticanonical divisor in $\X$. Therefore we saw that there is an up to scalar unique non-vanishing logarithmic $N$-form on $\XXcan$ which we called $\omegaCan$ (see Equations \eqref{e:omegaCanOdd} and \eqref{e:omegaCanEven}). This is the same form as the one appearing in \cite[Lemma~5.14]{GHK11}, and it also agrees with the one from \cite{rietsch} after the isomorphism of $\XXcan$ with $\RR$. It determines an element $[\omegaCan]$ in $M_B$.

\subsection{The case of odd-dimensional quadrics}

For odd-dimensional quadrics we recall the isomorphism between the $D$-modules on the two sides, proved using results from \cite{GS}. 

\begin{theorem}[{\cite[Corollary~13]{PR2}}]\label{t:connections-odd}
 	For $X=\Q_{2m-1}$ with  its mirror LG-model $(\XX_{2m-1}, \W)$ from Theorem~\ref{t:odd-quadric}, the map
	\[
		\begin{array}{ccc}
 			M_A&\to& M_B \\
			\sigma_i & \mapsto & [p_i \omegaCan]
		\end{array}
	\]
	defines an isomorphism of $D$-modules. \end{theorem}

\subsection{The case of even-dimensional quadrics}

For even quadrics $Q_{2m-2}$ we prove the  following.

\begin{theorem}\label{t:connections-even}
 	For $X=\Q_{2m-2}$ and the canonical mirror $(\XX_{2m-2}, \W)$, see \eqref{eq:Weven}, the map
	\[
		\begin{array}{clcl}
			\Psi:& M_A&\to& M_B \\
			&\sigma_i & \mapsto & [p_i \omegaCan] \\
			&\sigma_{m-1}' & \mapsto & [p_{m-1}' \omegaCan]
		\end{array}
	\]
	defines an injective homomorphism of $D$-modules. In particular, the $\C[\hbar^{\pm 1},q^{\pm 1}]$-submodule of $M_B$ generated by the classes $[p_i\omegaCan]$ and $[p_{m-1}'\omegaCan]$ is a submodule also for $D=\C[\hbar^{\pm 1},q^{\pm 1}]\langle \partial_{\hbar},\partial_{q}\rangle$. 
\end{theorem}

\begin{remark} 
	In the odd quadrics case, \cite{GS} (with N\'emethi and Sabbah) prove an additional property,  cohomological tameness, for the superpotential, which implies that the dimension of $M_B$ agrees with the number of critical points of $\W$. It is an interesting question whether this proof could be adapted to give a proof of cohomological tameness in the even case. Since by Proposition~\ref{p:critical} the number of critical points of $\W$ agrees with the dimension of $H^*(X,\C)$ this would imply that the injective homomorphism in Theorem~\ref{t:connections-even} is an isomorphism.  
\end{remark}

 To prove Theorem~\ref{t:connections-even} we consider a cluster algebra structure on our mirror $\XX_{2m-2}$. Cluster algebras were introduced by Fomin and Zelevinsky in the seminal paper \cite{FZ1}, which was the first of the series \cite{FZ1, FZ2,FZ3, FZ4}.

The coordinate ring $\C[\XX_{2m-2}]$ has a cluster algebra structure of type $A_1^{m-2}$ which 
is described in detail in \cite[Section 2]{GLS-Survey} and \cite[Section 12]{GLS-Partial}, and which we review here.  Note that the coordinates $\{y_1, y_2,\dots, y_{2m}\}$ in \cite{GLS-Survey, GLS-Partial}
correspond to our coordinates $\{\p_0, \p_1,\dots, \p_{m-2}, \p_{m-1}, \p'_{m-1}, \p_m,\dots, \p_{2m-2}\}$ here, while the coordinates $\{p_{\ell}\}$ in \cite{GLS-Survey, GLS-Partial} correspond to our coordinates $\{\delta_{\ell}\}$.

Consider the following initial quiver:
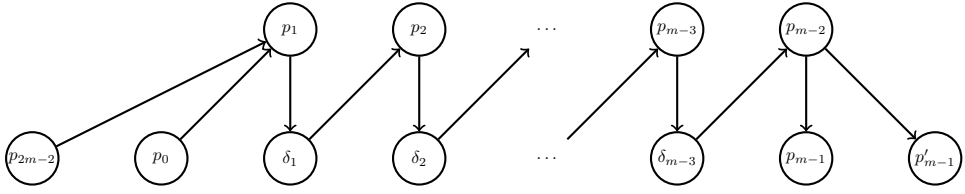
\begin{center}\label{fig:clusterquiver}
 \begin{tikzpicture}[every node/.style={thick,circle,inner sep=1pt,minimum size=30pt,scale=0.66},scale=0.66]
	\node[draw=black] (0,0) (1){$\p_1$};
	\node[draw=black,right=of 1] (2){$\p_2$};
	\node[draw=none,right=of 2] (3){$\dots$};
	\node[draw=black,right=of 3] (4){$\p_{m-3}$};
	\node[draw=black,right=of 4] (4bis){$\p_{m-2}$};
	\node[draw=black,below=of 1] (6){$\delta_1$};
	\node[draw=black,left=of 6] (5){$p_0$};
	\node[draw=black,left=of 5] (12){$p_{2m-2}$};
	\node[draw=black,right=of 6] (7){$\delta_2$};
	\node[draw=none,right=of 7] (8){$\dots$};
	\node[draw=black,right=of 8] (9){$\delta_{m-3}$};
	\node[draw=black,right=of 9] (10){$\p_{m-1}$};
	\node[draw=black,right=of 10] (11){$\p_{m-1}'$};
    \path (5) edge[->,thick] (1);
    \path (12) edge[->,thick] (1);
    \path (1) edge[->,thick] (6);
    \path (6) edge[->,thick] (2);
    \path (2) edge[->,thick] (7);
    \path (7) edge[->,thick] (3);
    \path (8) edge[->,thick] (4);
    \path (4) edge[->,thick] (9);
    \path (9) edge[->,thick] (4bis);
    \path (4bis) edge[->,thick] (10);
    \path (4bis) edge[->,thick] (11);
 \end{tikzpicture}
\end{center}
Here the initial cluster variables correspond to the vertices in the top row of the quiver, while the frozen variables (or coefficients) correspond to the vertices in the bottom row. Recall that the $p_i$  are  Pl\"ucker coordinates, and the $\delta_i$ are defined as in \eqref{eq:delta2}. We see from this description that the coordinate ring of $\XX_{2m-2}$ has a cluster structure of type $A_1^{m-2}$. In particular, it is of finite type, and there are $2^{m-2}$ different clusters, consisting of
\begin{itemize}
	\item the cluster variables $r_1,\dots,r_{m-2}$, where $r_i \in \{\p_i,\p_{2m-2-i}\}$;
	\item the frozen variables (or coefficients) $\delta_1,\dots,\delta_{m-3}$, $\p_0$, $\p_{m-1}$, $\p_{m-1}'$, and $\p_{2m-2}$.
\end{itemize}
The exchange relations are
\begin{align}\label{e:mutations}
	\p_i \p_{2m-2-i} = 	\begin{cases}
							\p_0 \p_{2m-2} + \delta_1 & \text{for $i=1$;}\\
							\delta_{i-1}+\delta_i & \text{for $1 \leq i \leq m-3$;} \\
							\delta_{m-3} + \p_{m-1}\p_{m-1}' & \text{for $i=m-2$}.
						\end{cases}
\end{align}
Note that the exchange relation for $i=m-2$ is a Plücker relation: it is the equation of the dual quadric \eqref{e:equationQuadric}.

\begin{remark} 
	In the case of $\X_{2m-3}^\circ$ the isomorphism with the Richardson variety combined with \cite{GLS} also gives a cluster algebra structure of type $A_1^{m-2}$, with a similar quiver to the one shown on page~\pageref{fig:clusterquiver} but where the frozen vertices labelled $p_{m-1}$ and $p'_{m-1}$ are identified.
\end{remark}

\begin{proof}[Proof of Theorem \ref{t:connections-even}]
For the injectivity of $\Psi$ we refer to \cite[Section~5]{MR}.
It remains to prove that $\Psi$ preserves the $D$-module structure. We use a change of coordinates to reduce the problem to checking only the action of $q\partial_q$. Namely, this follows by replacing $(\p_i,q,\hbar)$ with $(\mathbf p_i,\mathbf q,\hbar)$, where 
\[
	\mathbf p_i=\hbar^{-i}\p_i,\qquad \mathbf p_{m-1}'=\hbar^{1-m}\p_{m-1}',\qquad \mathbf q=\hbar^{-N}q, \qquad \hbar=\hbar,
\]
and observing that written in these coordinates the Gauss-Manin system for $\frac 1\hbar \W$ no longer involves the $\hbar$.  

Now we check that the map $\Psi$ preserves the action of $q\partial_q$. We consider the following identities in $QH^*(\Q_{2m-2},\C)$, which are a special case of results in \cite{FW}:

\begin{align}\label{e:quant-hyperplane}
	\sigma_1 \star_q \sigma_i = 	\begin{cases}
										\sigma_{i+1} & \text{for $0 \leq i \leq m-3$ or $m-1 \leq i \leq 2m-4$;} \\
										\sigma_{m-1}+\sigma_{m-1}' & \text{for $i=m-2$;} \\
										\sigma_{2m-2} + q \sigma_0 & \text{for $i=2m-3$;} \\
										q \sigma_1 & \text{for $i=2m-2$,}
								  \end{cases}
\end{align}	

\begin{equation}\label{e:quant-hyperplane2}
	\hspace{-3in} \sigma_1 \star_q \sigma_{m-1}' = \sigma_m.
\end{equation}

We need to prove that there are similar identities on the $B$ side:
\begin{align}\label{e:identities-W}
	\left[q\frac{\partial \W}{\partial q} \p_i \omegaCan\right] = 	\begin{cases}
																	[\p_{i+1} \omegaCan] & \text{for $0 \leq i \leq m-3$ or $m-1 \leq i \leq 2m-4$;} \\
																	[(\p_{m-1}+\p_{m-1}') \omegaCan] & \text{for $i=m-2$;} \\
																	[(\p_{2m-2} + q) \omegaCan] & \text{for $i=2m-3$;} \\
																	[q \p_1 \omegaCan] & \text{for $i=2m-2$,}
							 									\end{cases}
\end{align}

\begin{equation}\label{e:identities-W2}
	\left[q\frac{\partial \W}{\partial q} \p_{m-1}' \omegaCan \right] = [\p_m\omegaCan],
\end{equation}
where $\omegaCan$ is the canonical $(2m-2)$-form on $\XXcan$.

The proof of these identities in $M_B$ proceeds by constructing closed $(2m-3)$-forms $\nu_i$ and $\nu_{m-1}'$ such that the relation corresponding to $p_i$ will follow from the fact that 
\begin{equation*}
	[d \W\wedge \nu_i]=[( \hbar d +  d\W\wedge - )\nu_i]= 0 
\end{equation*}
and similarly for $p_{m-1}'$. (The first equality above comes from the fact that $\nu_i$ is closed,
and the second comes from the definition of $M_B$.)

Concretely, we will pick a cluster $\mathcal C$ containing a particular Pl\"ucker coordinate, say $p_i$, and use the following Ansatz for constructing $\nu_i$. We define a vector field
\begin{equation*}
	\xi_i= p_i \left(\sum_{c\in\mathcal C\setminus\{p_i\}} m_c c\partial_c\right) 
\end{equation*}  
and define an associated $(2m-3)$-form by insertion $\nu_i=\iota_{\xi_i}\omegaCan$, and analogously for $\nu_{m-1}'=\iota_{\xi_{m-1}'}\omegaCan$. Here the $m_c$'s are constants and $\iota$ is the interior product.

To see that these $(2m-3)$-forms are  closed, write $\omegaCan = \bigwedge_{p \in \mathcal C} \frac{dp}{p}.$
For $c\in \mathcal C$, we have $\iota_{c \partial_c} \omegaCan = \bigwedge_{p \in \mathcal C \setminus \{c\}} \frac{dp}{p},$ and so $\nu_i$ is a $\C$-linear combination of terms of the form $p_i \bigwedge_{p\in \mathcal C \setminus \{c\}} \frac{dp}{p}$ for $c \neq p_i$. Such a term is closed, because $p_i$ lies in $\mathcal C \setminus \{c\}$.

Using the fact that $d \W \wedge \omegaCan = 0$, we get $d \W \wedge \nu_i = \pm d \W(\xi_i) \omegaCan.$ It follows that 
\begin{equation*}
	d \W\wedge \nu_i=  p_i \left(\sum_{c\in\mathcal C\setminus\{p_i\}} m_c c \frac{\partial \W}{\partial c} \right)\omegaCan.  
\end{equation*}
Therefore e.g. in order to prove that $\left[q\frac{\partial \W}{\partial q} \p_i \omegaCan\right]-[p_{i+1} \omegaCan] = 0$, we will show that $q\frac{\partial \W}{\partial q} \p_i  - p_{i+1}$ has the form  $p_i \left(\sum_{c\in\mathcal C\setminus\{p_i\}} m_c c \frac{\partial \W}{\partial c} \right) $, for some choice of coefficients $m_c$.
  
To prove these identities, we will work with two clusters:
\begin{itemize}
	\item the initial cluster  
		$\mathcal{C}_1 = \{ p_1,\dots,p_{m-2},\delta_1,\dots,\delta_{m-3},\p_0, \p_{m-1},\p_{m-1}', \p_{2m-2}\}$;
	\item the cluster $\mathcal{C}_2 = 
		\{ p_{2m-3},\dots,p_m,\delta_1,\dots,\delta_{m-3},\p_0, \p_{m-1},\p_{m-1}', \p_{2m-2}\}$.
\end{itemize}

Let us first start with $\mathcal{C}_1$ and express $\W$ in terms of it using the exchange relations \eqref{e:mutations}. To simply our calculations, we set $\p_0=1$, and let $\delta_0$ denote $\p_0 \p_{2m-2} = \p_{2m-2}$.
\begin{align*}\label{e:W-chart1}
	\W =  \p_1 & + \sum_{\ell=1}^{m-3} \left( \frac{\p_{\ell+1}\delta_{\ell-1}}{\p_\ell \delta_\ell} + \frac{\p_{\ell+1}}{\p_\ell} \right) + \frac{\delta_{m-3}}{\p_{m-2}\p_{m-1}} + \frac{\delta_{m-3}}{\p_{m-2}\p_{m-1}'} \\
	& + \frac{\p_{m-1}}{\p_{m-2}} + \frac{\p_{m-1}'}{\p_{m-2}} + q \frac{\p_1}{\delta_0}.\notag
\end{align*}

The partial derivatives of $\W$ are:
\begin{align*}
	q\frac{\partial \W}{\partial q} &= q\frac{p_1}{\delta_0},\\
\p_1 \frac{\partial \W}{\partial \p_1} &= \p_1-\frac{\p_2\delta_0}{\p_1\delta_1}-\frac{\p_2}{\p_1}+q\frac{\p_1}{\delta_0}, \\
\p_i \frac{\partial \W}{\partial \p_i} &= \frac{\p_i\delta_{i-2}}{\p_{i-1}\delta_{i-1}}+\frac{\p_i}{\p_{i-1}}-\frac{\p_{i+1}\delta_{i-1}}{\p_i\delta_i}-\frac{\p_{i+1}}{\p_i} \text{ for $2 \leq i \leq m-3$}, \\
	\p_{m-2} \frac{\partial \W}{\partial \p_{m-2}} &= \frac{\p_{m-2}\delta_{m-4}}{\p_{m-3}\delta_{m-3}} + \frac{\p_{m-2}}{\p_{m-3}}- \frac{\delta_{m-3}}{\p_{m-2}\p_{m-1}} - \frac{\delta_{m-3}}{\p_{m-2}\p_{m-1}'}-\frac{p_{m-1}}{\p_{m-2}}-\frac{p_{m-1}'}{\p_{m-2}}, \\
\delta_0 \frac{\partial \W}{\partial \delta_0} &= \frac{\p_2\delta_0}{\p_1\delta_1} - q\frac{\p_1}{\delta_0}, \\
	\delta_i \frac{\partial \W}{\partial \delta_i} &= -\frac{p_{i+1}\delta_{i-1}}{\p_i\delta_i} + \frac{\p_{i+2}\delta_i}{\p_{i+1}\delta_{i+1}} \text{ for $1 \leq i \leq m-4$}, \\
\delta_{m-3} \frac{\partial \W}{\partial \delta_{m-3}} &= -\frac{\p_{m-2}\delta_{m-4}}{\p_{m-3}\delta_{m-3}} + \frac{\delta_{m-3}}{\p_{m-2}\p_{m-1}} + \frac{\delta_{m-3}}{\p_{m-2}\p_{m-1}'}, \\
	\p_{m-1}\frac{\partial \W}{\partial \p_{m-1}} &= -\frac{\delta_{m-3}}{\p_{m-2}\p_{m-1}} - \frac{\delta_{m-3}}{\p_{m-2}\p_{m-1}'} + \frac{\p_{m-1}}{\p_{m-2}},
\text{ and } \\
	\p_{m-1}'\frac{\partial \W}{\partial \p_{m-1}'} &= -\frac{\delta_{m-3}}{\p_{m-2}\p_{m-1}} - \frac{\delta_{m-3}}{\p_{m-2}\p_{m-1}'} + \frac{\p_{m-1}'}{\p_{m-2}}.
\end{align*}

Hence 
\begin{align*}
	q \frac{\partial \W}{\partial q} p_{i} - p_{i+1} &=  -\p_{i} \left(
\sum_{j=i+1}^{m-1} \p_j \frac{\partial \W}{\partial \p_j} + \p_{m-1}' \frac{\partial \W}{\partial \p_{m-1}'} + \sum_{j=0}^{m-3} \delta_j \frac{\partial \W}{\partial \delta_j} + \sum_{j=i}^{m-3} \delta_j \frac{\partial \W}{\partial \delta_j} \right) \\
	& \text{ for $0 \leq i \leq m-3$,} \notag \text{ and }\\
	q \frac{\partial \W}{\partial q} p_{m-2} - (p_{m-1}+p'_{m-1})&=  
 -\p_{m-2} \left(\p_{m-1} \frac{\partial \W}{\partial \p_{m-1}} +\p_{m-1}' \frac{\partial \W}{\partial \p_{m-1}'} + \sum_{j=0}^{m-3} \delta_j \frac{\partial \W}{\partial \delta_j} \right).
\end{align*}

Since the right-hand sides of the above equations have the form $p_i \left(\sum_{c\in\mathcal C\setminus\{p_i\}} m_c c\partial_c \W\right) $, this proves identity \eqref{e:identities-W} for $0 \leq i \leq m-2$.

To prove the remaining identities, we use the cluster  $\mathcal{C}_2$. In this cluster chart, $\W$ takes the following form:
\begin{align*}
	\W = & \frac{\delta_0}{\p_{2m-3}} + \frac{\delta_1}{\p_{2m-3}} + \sum_{\ell=1}^{m-4} \left( \frac{\p_{2m-2-\ell}}{\p_{2m-3-\ell}} + \frac{\p_{2m-2-\ell}\delta_{\ell+1}}{\p_{2m-3-\ell}\delta_\ell}\right) + \frac{\p_m}{\p_{m-1}} \\
	&  + \frac{\p_m}{\p_{m-1}'} + \frac{\p_{m+1}}{\p_m} + \frac{\p_{m-1}\p_{m-1}'\p_{m+1}}{\p_m\delta_{m-3}}+\frac{q}{\p_{2m-3}} + q\frac{\delta_1}{\p_{2m-3}\delta_0}. \notag
\end{align*}

Working out the partial derivatives of $\W$ as before, we get
\begin{align}
	q \frac{\partial \W}{\partial q} p_{m-1} - p_m &=
p_{m-1} \left( \p_{m-1}' \frac{\partial \W}{\partial \p_{m-1}'} + \sum_{j=m}^{2m-3} \p_j\frac{\partial \W}{\partial \p_j} + \sum_{j=0}^{m-3} \delta_j \frac{\partial \W}{\partial \delta_j}\right)  \\
	q \frac{\partial \W}{\partial q} p'_{m-1} - p_m &= p'_{m-1} \left(
\p_{m-1} \frac{\partial \W}{\partial \p_{m-1}} + \sum_{j=m}^{2m-3} \p_j\frac{\partial \W}{\partial \p_j} + \sum_{j=0}^{m-3} \delta_j \frac{\partial \W}{\partial \delta_j} \right)\\
	q \frac{\partial \W}{\partial q} p_i -p_{i+1} &= p_i \left(
-\sum_{j=i+1}^{2m-3} \p_j\frac{\partial \W}{\partial \p_j} - \sum_{j=0}^{2m-3-i} \delta_j \frac{\partial \W}{\partial \delta_j} \right) \\
	& \qquad\qquad \text{for $m \leq i \leq 2m-4$}, \notag 
\end{align}

Recall that $\delta_0$ is $p_{2m-2}$. The final two relations are
\begin{align}
 	q\frac{\partial \W}{\partial q} \p_{2m-3}  -   (p_{2m-2}+q) &= -p_{2m-3}\delta_0 \frac{\partial \W}{\partial \delta_0}
 \text{\quad and } \\
	q \frac{\partial \W}{\partial q}\p_{2m-2} - q p_1 &=0
\end{align}
This gives us the identities \eqref{e:identities-W} for $m-1 \leq i \leq 2m-2$, as well as  \eqref{e:identities-W2}.
\end{proof}

\section{The hypergeometric flat section of a quadric}\label{s:aseries}

Givental in \cite{Givental:EqGW} constructed flat sections of a dual version of the Dubrovin connection (see Equations \eqref{e:dualDubrovin1} and \eqref{e:dualDubrovin2} below) in terms of Gromov-Witten invariants. 
In this section we directly and explicitly compute all the components of a distinguished flat section and the resulting invariants, in two different ways. The first component we consider is also a particular component of Givental's $J$-function.  

\subsection{The dual Dubrovin connection and the $J$-function}

We begin by defining Givental's $J$-function and what we call the `quantum differential operators'. Consider the dual connection to ${}^A\nabla$ with respect to the pairing 
\[
	\langle \sigma ,\tau \rangle=(2\pi i \hbar)^N\int_{X} \sigma\cup \tau.
\]
Here $\sigma \cup \tau$ is the usual cup product of $\sigma$ and $\tau$, which we will subsequently also denote by $\sigma \tau$. Explicitly, the dual connection is given by the formulas:
\begin{eqnarray}\label{e:dualDubrovin1}
	{}^A\nabla^\vee_{q\partial_q} S&:=& q\frac{\partial S}{\partial q} - \frac{1}{\hbar} \sigma_{1}\star_{q}S, \\
	\label{e:dualDubrovin2}
	{}^A \nabla^\vee_{\hbar\partial_\hbar} S&:=& \hbar\frac{\partial S}{\partial\hbar} + \frac{1}{\hbar} c_1(TX)\star_{q}S+ \operatorname{Gr} (S),
\end{eqnarray}
compare~\cite[Definition~3.1]{iritani}. For the purposes of the $J$-function we ignore the ${}^A \nabla^\vee_{\hbar\partial_\hbar}$ part of the covariant derivative and consider  ${}^A\nabla^\vee_{q\partial_q} $ as a family of connections (in the parameter $\hbar$). Formal flat sections  indexed by the cohomology basis were written down by Givental \cite{Givental:EqGW} in terms of descendent Gromov-Witten invariants. We denote these sections by $S_0,\dotsc, S_{2m-1}$ in the case of $Q_{2m-1}$, and  by $S_0,\dotsc, S_{m-1},S_{m-1}',S_m,\cdots, S_{2m-2}$ for $Q_{2m-2}$, in keeping with the notation from \eqref{e:SchubertGS} for Schubert classes. 
See \cite[(10.14)]{CoxKatz} for a precise definition of the sections $S_i$.

\begin{defn} We define Givental's $J$-function in our setting as 
\[
	J=(2\pi i \hbar)^N\sum \langle S_j,\sigma_0\rangle \sigma_{PD(j)},
\]
where the sum  is over all the Schubert classes, including $\sigma_{m-1}'$ in the even case, and where  $\sigma_{PD(j)}$ stands for the Poincar\'e dual cohomology class to $\sigma_j$. 
\end{defn}

In the case of a quadric (or, indeed, of any projective Fano complete intersection), the $J$-function is computed explicitly in ~\cite[Theorem 9.1]{Givental:EqGW} from the $J$-function of projective space. Namely
\begin{equation}\label{e:Givental}
	J^{Q_N}=e^{\frac{\ln(q)\sigma_1}{\hbar}}\sum_{d\ge 0}\frac{\prod_{j=1}^{2d}(2\sigma_1+j\hbar)}{\prod_{j=1}^d(\sigma_1+j\hbar)^{N}}q^d.
\end{equation}

We consider a family of differential operators which annihilate the $J$-function:
\begin{defn}[{\cite[Definition~10.3.2]{CoxKatz}}]
The differential operators $P$ which are formal power series in $\hbar q\partial_q, q, \hbar$ and which annihilate the coefficients of Givental's $J$-function are called \emph{quantum differential operators}.
\end{defn}

\subsection{The hypergeometric term of the $J$-function.}


Among Givental's flat sections $S_i$, the flat section $S_N$ corresponding to the class of a point has the property that all its coefficients are power series in ${\bf q}=\hbar^{-N}q$.
Moreover, a special role is played by the coefficient $(2\pi i\hbar )^N\langle S_N, \sigma_0\rangle$, also appearing as the coefficient of the fundamental class in the definition of $J$-function. We define it as in \cite[Definition~5.1.1]{BCFKvSGr}: 
\begin{defn}
	The \emph{hypergeometric series} $A_X$ of $X$ is the unique power series  of the form $A_X (q)= 1+\sum_{k=1}^{\infty} a_k { q}^k$, for which $P({q} \partial_{q},{q},1)A_X=0$ for all quantum differential operators $P(\hbar q\partial_q,q,\hbar)$ specialized to $\hbar=1$. We denote the hypergeometric series $A_{Q_N}$ of the quadric $Q_N$ by $A_N$.
\end{defn}

The hypergeometric series $A_N$ of the quadric $Q_N$ may be obtained by setting $\hbar$ to $1$ in $(2\pi i\hbar )^N\langle S_N,\sigma_0\rangle$. Alternatively we have $\langle S_N,\sigma_0\rangle=A_N(\hbar^{-N}q)$. 

We recall the geometric interpretation of the coefficients of $A_X$ below. The flat sections $S_i$ and in particular the $J$-function encode certain descendent Gromov-Witten invariants. Let 
\begin{equation}\label{Gromov-Witten}
	I_k(\psi_1^{a_1}\gamma_1,\dots, \psi_r^{a_r} \gamma_r)
\end{equation}
denote the degree $k$ descendent Gromov-Witten invariant associated to the cohomology classes $\gamma_1, \dots, \gamma_r$, where the $\psi$-class $\psi_i$ denotes the first Chern class of the $i$th cotangent
bundle  of the moduli space of degree $k$ genus $0$ stable maps with $r$ marked points, see
\cite[Section~10.1]{CoxKatz}. Let $\psi$ stand for $\psi_1$. If we write
\[
	J^{\Q_N} = (2\pi i \hbar)^N \sum J^{\Q_N}_i \sigma_{\PD(i)},
\]
then in fact $A_N(\hbar^{-N} q)=\langle S_N,\sigma_0\rangle=J^{\Q_N}_N $ and we have
\begin{align*}
A_N(\hbar^{-N} q)		&= J^{\Q_N}_N =1 + \sum_{k=1}^\infty q^k I_k\left(\frac{\sigma_N e^{\frac{\ln(q)\sigma_1}{\hbar}}}{\hbar-\psi},\sigma_0\right) \\
				&= 1 + \sum_{k=1}^\infty \sum_{j=0}^\infty \sum_{i=0}^\infty \frac{q^k}{\hbar} I_k\left(\sigma_N \left(\frac{\ln(q)\sigma_1}{\hbar}\right)^j\frac{1}{j!}\left(\frac{\psi}{\hbar}\right)^i,\sigma_0\right).
\end{align*}
The cup-product $\sigma_N\cup \left(\frac{\ln(q)\sigma_1}{\hbar}\right)^j$ is nonzero if and only if $j=0$. Therefore we have
\[
	J^{\Q_N}_N = 1 + \sum_{k=1}^\infty \sum_{i=0}^\infty \frac{q^k}{\hbar} I_k\left( \sigma_N \left(\frac{\psi}{\hbar}\right)^i,\sigma_0\right).
\]
Now  the dimension of the moduli space of stable maps $\overline{\mathcal{M}}_{0,2}(\Q_N,k)$ is equal to $(k+1)N-1$, hence
\[
	J^{\Q_N}_N = 1 + \sum_{k=1}^\infty \frac{q^k}{\hbar} I_k\bigg( \sigma_N \left(\frac{\psi}{\hbar}\right)^{k N-1},\sigma_0\bigg).
\]
Next we use the fundamental class axiom to get
\[
	J^{\Q_N}_N = 1 + \sum_{k=1}^\infty \left(\frac{q}{\hbar^N}\right)^k I_k\left( \sigma_N \psi^{k N-2}\right).
\]
If we set $\hbar=1$ in $J^{\Q_N}_N$, this gives exactly the hypergeometric series of the quadric, since $J^{\Q_N}_N=A_{N}(\hbar^{-N}q)$. Hence we obtain the following geometric interpretation of the coefficient $a_k$ of $q^k$ in $A_N(q)$:
\begin{equation}\label{e:AseriesGW}
	a_k = I_k\left( \sigma_N \psi^{k N-2}\right).
\end{equation}

\subsection{The hypergeometric flat section of the dual Dubrovin connection}
In this Section, as an illustration of the mirror theorem, we compute explicitly the coefficients of the hypergeometric flat section $S_N$ of the Dubrovin connection for $Q_N$, once using the $A$-model and once using the $B$-model. The main result of the computations is the following.

\begin{theorem}\label{thm:hypergeometric}
The hypergeometric flat section $S_N$ of the dual Dubrovin connection for $Q_N$ is given by the expansion 
\[
S_N=\frac{1}{(2\pi i\hbar )^N}{\sum_{\ell=0}^{N}}{}'\langle S_N,\sigma_{\ell}\rangle \sigma_{PD(\ell)},
\] 
where $\sum'$ means that we add an extra summand $\langle S_N,\sigma_{m-1}'\rangle \sigma_{m-1}$ when $N=2m-2$. The coefficients are given by the following formulas:	
	\begin{equation*}
\langle S_N,\sigma_{\ell}\rangle=	\begin{cases}		
	\sum_{k \geq 0}  \frac{k^\ell}{\hbar^{kN-\ell}(k!)^N}  \cdot \binom{2k}{k} \cdot q^k &\text{ if $0 \leq \ell \leq \lfloor \frac{N-1}{2} \rfloor$,} \\
		\sum_{k \geq 0} \frac{k^\ell}{2 \hbar^{kN-\ell}(k!)^N} \cdot \binom{2k}{k} \cdot q^k &\text{ if $\lfloor \frac{N+1}{2} \rfloor \leq \ell \leq N-1$,} \\
		\sum_{k \geq 0} \frac{1}{\hbar^{(k-1)N}(k-1)!^N} \cdot \frac{k-1}{k} \cdot \binom{2k-2}{k-1} \cdot q^k &\text{ if $\ell=N$.} 
		\end{cases}
	\end{equation*}
	Moreover, when $N=2m-2$ is even, we have
	\[
		\langle S_N,\sigma_{m-1}'\rangle=\sum_{k \geq 0} \frac{k^{m-1}}{2 \hbar^{kN+1-m} (k!)^N} \cdot \binom{2k}{k} \cdot q^k.
	\]
\end{theorem}

The $\ell=0$ special case of Theorem \ref{thm:hypergeometric} gives the following.
\begin{cor}\label{thm:A}
	The hypergeometric series of the quadric $Q_N$ is 
	\begin{equation}\label{e:hypergeom}
		A_N(q)= 1 + \sum_{k \geq 1} \frac{1}{(k!)^N} \binom{2k}{k} q^k.
	\end{equation}
	The Gromov-Witten invariant $I_k(\sigma_N \psi^{N k-2})$ is given by
	\begin{equation}\label{e:GWformula}
		I_k( \sigma_N \psi^{N k-2}) = \frac{1}{(k!)^N} \binom{2k}{k}.
	\end{equation}
\end{cor}

This corollary is easily verified in the $A$-model.  The formula~\eqref{e:GWformula} follows from equations~\eqref{e:hypergeom} and 
   \eqref{e:AseriesGW}, 
while the second formula, \eqref{e:hypergeom}, follows easily from the formula \eqref{e:Givental} for the $J$-function of $Q_N$. In the odd quadrics case the $D$-module is cyclic and hence the constant term determines all of the other terms of the flat section. However for even quadrics this is not the case.
We now give a direct A-model proof of Theorem~\ref{thm:hypergeometric} which works in the even and odd case alike. 

\begin{proof}[A-model proof]
	Our A-model proof works by recovering Theorem~\ref{thm:hypergeometric} from the recurrence relations of Kontsevich-Manin for Gromov-Witten invariants \cite{KM}. Define
	\begin{equation}\label{e:descGW}
		\beta_{\ell,k} = I_k( \psi^{Nk-1-\ell} \sigma_N,\sigma_\ell).
	\end{equation}

	Let us first assume that $N=2m-1$. Using the divisor axiom and topological recursion, we get:
	\begin{align*}
		k \beta_{\ell,k} = I_k(\psi^{Nk-1-\ell} \sigma_N,\sigma_\ell,\sigma_1) = \begin{cases}
																			\beta_{\ell+1,k} & \text{if 			$\ell\not\in\{m-1,N-1,N\}$,} \\
																			2\beta_{m,k} & \text{if 				$\ell=m-1$,} \\
																			\beta_{N,k}+\beta_{0,k-1} & 			\text{if $\ell=N-1$,} \\
																			\beta_{1,k-1} & \text{if 				$\ell=N$.}																		\end{cases}
\end{align*}
A straightforward computation then gives 
\begin{align*}
	\frac{\beta_{\ell,k+1}}{\beta_{\ell,k}}=	\begin{cases}
													\frac{2 (2k+1)}{k^{\ell}(k+1)^{N+1-\ell}}&\text{ if $0\leq \ell\leq N-1$,}\\
													\frac{ 2 (2k-1)}{(k-1) k^{N-1}(k+1)}&\text {if $\ell= N$,}\\
												\end{cases}
\end{align*}
and $\beta_{1,1}=2$, which yields Theorem~\ref{thm:hypergeometric}.
\qed

Similarly, in the case where $N=2m-2$:
\begin{align*}
k \beta_{\ell,k} = I_k(\psi^{Nk-1-\ell} \sigma_N,\sigma_\ell,\sigma_1) = \begin{cases}
																			\beta_{\ell+1,k} & \text{if $\ell\not\in\{m-2,N-1,N\}$,} \\
																			\beta_{m-1,k}+\beta_{m-1,k}' & \text{if $\ell=m-2$,} \\
																			\beta_{N,k}+\beta_{0,k-1} & \text{if $\ell=N-1$,} \\
																			\beta_{1,k-1} & \text{if $\ell=N$,}
																			\end{cases}
\end{align*}
and 
\[
	k\beta_{m-1,k}'=\beta_{m,k}.
\]
Theorem~\ref{thm:hypergeometric} is then easily checked.
\end{proof}

\begin{proof}[B-model proof]
We consider the distinguished flat section of the Dubrovin connection whose coefficients are expressed in terms of the $B$-model as residue integrals, see Section~\ref{s:applications} and compare with \cite[Theorem 4.2]{MR}. Explicitly, we let $\Gamma_0\cong (S^1)^N$ be a compact cycle  inside $\XXcan$ such that $\int_{\Gamma_0}\omegaCan=1$. Then the integral formula
\begin{equation}\label{e:SGamma0}
S_{\Gamma_0}(\hbar,q) := \frac{1}{(2\pi i \hbar)^N}\sum \left(\int_{\Gamma_0}e^{\frac 1\hbar \W} p_i\omegaCan\right) \sigma_{N-i}
\end{equation}
defines a flat section of the Dubrovin connection in the $N=2m-1$ case, and with $(\int_{\Gamma_0}e^{\frac 1\hbar \W} p_{m-1}\omegaCan) \sigma_{m-1}$ replaced by
$(\int_{\Gamma_0}e^{\frac 1\hbar \W} p_{m-1}'\omegaCan) \sigma_{m-1}+ (\int_{\Gamma_0}e^{\frac 1\hbar \W} p_{m-1}\omegaCan) \sigma_{m-1}'$ in the $N=2m-2$ case. 

We will prove the formula in Theorem~\ref{thm:hypergeometric} in one representative case, but omit the other cases, which are extremely similar.

Let us consider the case that $N=2m-2$, and $m \leq \ell \leq 2m-3$. In this case recall that 
$p_{\ell} = a_1 \dots a_{m-2} cd b_{m-2} \dots b_{2m-1-\ell}$, and
recall from \eqref{e:quiverMirrorEven} that the superpotential $\W$ equals 
\begin{small}
  \begin{equation*}
      a_1 + \dots + a_{m-2} + c +d+  b_{m-2} + \dots + b_1 + 
 \frac{q}{a_2 \dots a_{m-2} cd b_{m-2} \dots b_1} +
 \frac{q}{a_1 \dots a_{m-2} cd  b_{m-2} \dots b_2}
  \end{equation*}
\end{small}
in terms of the usual coordinates on $\XXlus$ viewed as a torus chart in $\XXcan$.

To compute the constant term of $p_{\ell} \exp(\frac{1}{\hbar} \W)$, we consider 
\[
	p_{\ell}\left(1+\frac{1}{\hbar} \W+ \frac{1}{\hbar^2} 
\frac{\W^2}{2!} + \frac{1}{\hbar^3} \frac{\W^3}{3!} + \dots \right),
\]
and we pick out from each  $p_{\ell} \frac{\W^i}{\hbar^i i!}$ every term which has the form $\lambda q^j$ where $\lambda \in \mathbb{Q}[\frac{1}{\hbar}]$. Here we just need to look at each 
$\frac{\W^{kN-\ell}}{\hbar^{kN-\ell} (kN-\ell)!}$ for $k=1,2,\dots,$ because the expansion of $p_{\ell} \frac{\W^i}{\hbar^i i!}$ for $i$ not of the form $kN-\ell$
will contain no terms of the form $\lambda q^j$ for $\lambda \in \mathbb{Q}[\frac{1}{\hbar}]$.

Now let us analyze $p_{\ell} \frac{\W^{kN-\ell}}{\hbar^{kN-\ell} (kN-\ell)!}$ for $N = 2m-2$.  A (Laurent) monomial in the expansion of $p_{\ell} \W^{k(2m-2)-\ell}$ is obtained by choosing one term in each of the $k(2m-2)-\ell$ factors. Some of the monomials in the expansion will be pure in the variable
$q$ alone -- in which case they will equal $q^k$. We need to show that the number of such monomials 
divided by $(k(2m-2)-\ell)!$ equals $\frac{1}{2} \binom{2k}{k} k^{\ell} /(k!)^{k(2m-2)}$. To count the number of such monomials, we need to pick one term in each of the $k(2m-2)-\ell$ factors so that we:
\begin{itemize}
	\item choose $i$ terms which are $\frac{q}{a_2 \dots a_{m-2} cd b_{m-2} \dots b_1}$ for some $0 \leq i \leq k$;
	\item choose $k-i$ terms which are $\frac{q}{a_1 \dots a_{m-2} cd b_{m-2} \dots b_2}$;
	\item choose $k-1$ terms which are $c$;
	\item choose $k-1$ terms which are $d$;
	\item choose $i$ terms which are $b_1$;
	\item choose $k-i-1$ terms which are $a_1$;
	\item for each $j$ such that $2 \leq j \leq m-2$, choose $k-1$ terms which are $a_j$;
	\item for each $j$ such that $2 \leq j \leq 2m-2-\ell$, choose $k$ terms which are $b_j$.
	\item for each $j$ such that $2m-2-\ell < j \leq m-2$, choose $k-1$ terms which are $b_j$.
\end{itemize}
The number of ways to do this is the sum of multinomial coefficients
\begin{equation}\label{eq:multinomial}
	\sum_{i=0}^k \binom{k(2m-2)-\ell}{i, i, k-i, k-i-1, k\dots k, k-1 \dots k-1},
\end{equation}
where the number of $k$'s in the string $k \dots k$ above is $2m-2-\ell-1$, and the number of $k-1$'s in the string $k-1 \dots k-1$ above is $\ell-1$. When we simplify \eqref{eq:multinomial} and divide by 
$(k(2m-2)-\ell)!$, we obtain $\frac{1}{2} \binom{2k}{k} k^{\ell} /(k!)^{k(2m-2)}$, as desired.
\end{proof}

\bibliographystyle{amsalpha}
\bibliography{biblio}

\end{document}